\documentclass[12pt,leqno]{article}
\usepackage{amssymb,amsmath,amsthm}
\usepackage[all]{xy}
\usepackage{dhbar}

\numberwithin{equation}{section}

\begin{document}

\title{Regular holonomic $\shd\forl$-modules}

\author{Andrea D'Agnolo, St{\'e}phane Guillermou and Pierre Schapira}

\date{}

\maketitle

\begin{abstract}
We describe the category of regular holonomic modules over the ring
$\shd\forl$ of linear differential operators with a formal parameter
$\h$.  In particular, we establish the Riemann-Hilbert
correspondence and discuss the additional $t$-structure related to
$\h$-torsion.
\end{abstract}

\section*{Introduction}
On a complex manifold $X$, we will be interested in the study of
holonomic modules over the ring $\Dhh$ of 
differential operators with a formal parameter $\h$.
Such modules naturally appear when studying deformation quantization
modules ($\DQ$-modules) along a smooth Lagrangian submanifold of a
complex symplectic manifold (see \cite[Chapter~7]{KS08}).

In this paper, after recalling the tools from loc.\ cit.\ that we
shall use, we explain some basic notions of $\Dhh$-modules
theory. For example, it follows easily from general results on modules
over
$\C\forl$-algebras that given two holonomic $\Dhh$-modules
$\shm$ and $\shn$, the complex $\rhom[\Dhh](\shm,\shn)$ is
constructible over $\C\forl$ and the microsupport of
the solution complex $\rhom[\Dhh](\shm,\Ohh)$ coincides with the
characteristic variety of $\shm$.

Then we establish our main result, the Riemann-Hilbert correspondence
for regular holonomic $\Dhh$-modules, an $\h$-variant of
Kashiwara's classical theorem. 
In other words, we show that the solution functor with values in
$\Ohh$ induces an equivalence between the derived category of
regular holonomic $\Dhh$-modules and that of constructible sheaves
over $\C\forl$. A quasi-inverse is obtained by 
constructing the ``sheaf'' of holomorphic functions with temperate
growth and a formal parameter $\h$ in the subanalytic site. This
needs some care since the
literature on this subject is written in the framework of sheaves
over a field and does not immediately apply to the ring $\C\forl$. 

We also discuss the $t$-structure related to $\h$-torsion. Indeed,
as we work over the ring $\C\forl$ and not over a field, 
the derived category of holonomic $\Dhh$-modules (or,
equivalently, that of constructible sheaves over $\C\forl$)
has an additional $t$-structure related to $\h$-torsion.
We will show how the duality functor interchanges it with the natural
$t$-structure. 

Finally, we describe some natural links between the ring
$\Dhh$ and deformation quantization algebras, as mentioned above.

\section*{Notations and conventions}

We shall mainly follow the notations of \cite{KS06}.
In particular, if $\shc$ 
is an abelian category, we denote by 
$\Der(\shc)$ the derived category of $\shc$ 
and by $\Der[*](\shc)$ ($*=+,-,\rmb$) the full triangulated subcategory
consisting of
objects with  bounded from below (resp.\ bounded from above, resp.\
bounded)
cohomology.

For a sheaf of rings $\shr$ on a topological space, or more generally a
site, we denote by
$\md[\shr]$ the category of left $\shr$-modules and we write 
$\Der[*](\shr)$ instead of $\Der[*](\md[\shr])$ ($*=\emptyset,+,-,\rmb$). 
We denote by $\mdc[\shr]$ 
the full abelian subcategory of $\md[\shr]$ of coherent objects,
and by $\Derb_\coh(\shr)$ the full triangulated subcategory 
of $\Derb(\shr)$ of objects with coherent cohomology groups.

If $R$ is a ring (a sheaf of rings over a point), we write for short $\Derb_f(R)$ instead of
$\Derb_\coh(R)$.  

\section{Formal deformations (after~\cite{KS08})}
We review here some definitions and results from \cite{KS08}
that we shall use in this paper.

\subsubsection*{Modules over $\Z[\hbar]$-algebras}

One says that a $\Z[\hbar]$-module $\shm$ has no $\h$-torsion if 
$\h\cl\shm\to\shm$ 
is injective and one says that $\shm$ is $\h$-complete if
$\shm\to\sprolim[n] \shm/\h^n\shm$ is an isomorphism.

\medskip
Let $\shr$ be a $\Z[\hbar]$-algebra, and assume that $\shr$ has no $\h$-torsion.
One sets 
\[
\shr^\loc\eqdot \Z[\hbar,\hbar^{-1}]\tens[{\Z[\hbar]}]\shr,
\qquad\shro\eqdot \shr/\hbar\shr,
\]
and considers the functors
\begin{align*}
(\scbul)^\loc&\cl\md[\shr]\to\md[\shr^\loc],\quad
\shm\mapsto\shm^\loc\eqdot\shr^\loc\tens[{\shr}]\shm,\\
\gr&\cl \Der(\shr)\to\Der(\shro),\quad
\shm\mapsto\gr(\shm)\eqdot \shro\lltens[\shr]\shm.
\end{align*}
Note that $(\scbul)^\loc$ is exact and that for $\shm,\shn\in\Derb(\shr)$ and $\shp\in\Derb(\shr^\op)$ one
has isomorphisms:
\eq\label{eq:grFtens}
\gr(\shp\ltens[\shr]\shm) &\simeq& \gr\shp\ltens[\shro]\gr\shm, \\
\label{eq:grF}
\gr(\rhom[\shr](\shm,\shn)) &\simeq& \rhom[\shro](\gr(\shm),\gr(\shn)).
\eneq

\subsubsection*{Cohomologically $\h$-complete sheaves}

\begin{definition}\label{def:cohco}
One says that  an object $\shm$ of $\Der(\shr)$ 
is cohomologically $\h$-complete if 
$\rhom[\shr](\shr^\loc,\shm)= 0$.
\end{definition}
Hence, the full subcategory of cohomologically $\h$-complete objects is
triangulated. In fact, it is the right orthogonal to the full
subcategory $\Der(\shr^\loc)$ of $\Der(\shr)$.

Remark that $\shm\in\Der(\shr)$ is cohomologically $\h$-complete if and only
if its image in
$\Der(\Z_X[\hbar])$ is cohomologically $\h$-complete.
\begin{proposition}\label{pro:cohco1} 
Let $\shm\in\Der(\shr)$. Then $\shm$ is  cohomologically $\h$-complete if
and only if
\eqn
&&\indlim[U\ni
x]\Ext[{\Z[\hbar]}]{j}\bl\Z[\hbar,\hbar^{-1}],H^i(U;\shm)\br= 0,
\eneqn
for any $x\in X$, any integer $i\in\Z$ and any $j=0,1$.
Here, $U$ ranges over an open neighborhood system of $x$.
\end{proposition}

\begin{corollary}\label{cor:cohcosheaves}
Let $\shm\in\md[\shr]$. Assume that $\shm$ has no $\h$-torsion, is
$\h$-complete and there exists a base $\BB$ of open subsets such
that $H^i(U;\shm)= 0$ for any $i>0$ and any $U\in\BB$. Then $\shm$
is cohomologically $\h$-complete.
\end{corollary}

The functor $\gr$ is conservative on the category of cohomologically
$\h$-complete objects:
\begin{proposition}\label{pro:conserv}
Let $\shm\in\Der(\shr)$ be a cohomologically $\h$-complete object. If
$\gr(\shm)= 0$, then $\shm= 0$.
\end{proposition}

\begin{proposition}\label{pro:homcc}
Assume that  $\shm\in\Der(\shr)$ is cohomologically $\h$-complete. Then
$\rhom[\shr](\shn,\shm)\in\Der(\Z_X[\hbar])$ is cohomologically $\h$-complete
for any $\shn\in\Der(\shr)$.
\end{proposition}

\begin{proposition}\label{pro:cohcodirim}
Let $f\cl X\to Y$ be a continuous map, and $\shm\in\Der(\Z_X[\hbar])$.
If $\shm$ is cohomologically $\h$-complete, then so is $\roim{f}\shm$.
\end{proposition}

\subsubsection*{Reductions to $\h=0$}
Now we assume that $X$ is a Hausdorff locally compact topological space.

By a basis $\BB$ of compact subsets 
of $X$, we mean a family of compact subsets such that for any $x\in X$ and
any open neighborhood $U$ of $x$, there exists $K\in\BB$ such that 
$x\in\Int(K)\subset U$.

Let $\sha$ be a $\Z[\hbar]$-algebra, and recall that we set $\shao = \sha/\h\sha$.
Consider the following conditions:
\begin{itemize}
\item[(i)] $\sha$ has no $\h$-torsion and is $\h$-complete,
\item[(ii)] $\shao$ is a left Noetherian ring,
\item[(iii)]
there exists a basis $\BB$ of compact subsets of $X$ and a prestack 
$U\mapsto\mdgd[\shao\vert_U]$ ($U$ open in $X$) such that 
\banum
\item for any $K\in \BB$ and an open subset $U$ such that $K\subset U$,
there exists $K'\in\BB$ such that
$K\subset\Int(K')\subset K'\subset U$,
\item $U\mapsto \mdgd[\shao\vert_U]$ is a full subprestack of 
$U\mapsto \mdcoh[\shao\vert_U]$,
\item 
for any $K\in\BB$, any open set $U$ containing $K$, any 
$\shm\in\mdgd[\shao\vert_U]$
and any $j>0$, one has $H^j(K;\shm)=0$,
\item for an open subset $U$ and $\shm\in\mdcoh[\shao\vert_U]$,
if $\shm\vert_V$ belongs to $\mdgd[\shao\vert_V]$ for any 
relatively compact open subset $V$ of $U$,
then $\shm$ belongs to $\mdgd[\shao\vert_U]$,\label{cond:exh}
\item
for any $U$ open in $X$, 
$\mdgd[\shao\vert_U]$ is stable by subobjects, quotients and extensions
in $\mdcoh[\shao\vert_U]$,
\item for any $\shm\in\mdcoh[\shao\vert_U]$, there exists an open
covering 
$U=\bigcup_iU_i$ such that
$\shm\vert_{U_i}\in\mdgd[\shao\vert_{U_i}]$,
\label{goodlocal}
\item $\shao\in\mdgd[\shao]$,\label{cond:good}
\eanum 
\item[(iii)']
there exists a basis $\BB$ of open subsets of $X$ such that 
for any $U\in\BB$, any 
$\shm\in\Mod_\coh(\shao\vert_U)$
and any $j>0$, one has $H^j(U;\shm)=0$.
\end{itemize}

We will suppose that $\sha$ and $\shao$ satisfy either Assumption~\ref{as:FDring} or Assumption~\ref{as:DQring}.

\begin{assumption}\label{as:FDring}
$\sha$ and $\shao$ satisfy conditions (i), (ii) and (iii) above.
\end{assumption}

\begin{assumption}\label{as:DQring}
$\sha$ and $\shao$ satisfy conditions (i), (ii) and (iii)' above.
\end{assumption}

\begin{theorem}\label{th:formalfini1}
\bnum
\item
$\sha$ is a left Noetherian ring.
\item 
Any coherent $\sha$-module $\shm$ is $\h$-complete.
\item 
Let $\shm\in \Derb_\coh(\sha)$. Then $\shm$ is cohomologically $\h$-complete.
\enum
\end{theorem}

\begin{corollary}\label{cor:conservative1}
The functor $\grh\cl\Derb_\coh(\sha)\to\Derb_\coh(\shao)$ is
conservative.
\end{corollary}

\begin{theorem}\label{th:formalfini2}
Let $\shm\in\Der[+](\sha)$ and assume:
\banum
\item[{\rm(a)}]
$\shm$ is cohomologically $\h$-complete,
\item[{\rm(b)}]
$\gr(\shm)\in\Der[+]_\coh(\shao)$.
\eanum
Then,  $\shm\in\Der[+]_\coh(\sha)$ and for all $i\in\Z$ we have the
isomorphism
\eqn
&&H^i(\shm)\isoto\prolim[n]H^i(\sha/\h^n\sha\ltens[\sha]\shm).
\eneqn
\end{theorem}

\begin{theorem}\label{th:flat2}
Assume that $\sha_0^\rop = \sha^\rop/\hbar\sha^\rop$ is a Noetherian ring
and the flabby dimension of $X$ is finite.
Let $\shm$ be an $\sha$-module.
Assume the following conditions:
\banum
\item $\shm$ has no $\h$-torsion,
\item $\shm$ is cohomologically $\h$-complete, 
\item $\shm/\hbar\shm$ is a flat $\shao$-module.
\eanum
Then $\shm$ is a flat $\sha$-module.

If moreover $\shm/\h\shm$ is a faithfully flat $\shao$-module, 
then $\shm$ is a faithfully flat $\sha$-module.
\end{theorem}

\begin{theorem}\label{th:hddim}
Let $d\in \N$. Assume that $\shao$ is $d$-syzygic, {\em i.e.,}~that
any coherent $\shao$-module locally admits 
a projective resolution of length $\leq d$ by free $\shao$-modules of
finite rank. 
Then
\banum
\item
$\sha$ is $(d+1)$-syzygic.
\item
Let $\shm^\scbul$ be a complex of $\sha$-modules concentrated in 
degrees $[a,b]$ and with coherent cohomology groups. 
Then, locally there exists a 
quasi-isomorphism
$\shl^\scbul\to \shm^\scbul$ where $\shl^\scbul$ is a complex of free
$\sha$-modules of finite rank concentrated in degrees $[a-d-1,b]$.
\eanum
\end{theorem}

\begin{proposition}\label{prop:cohom_and_grad}
Let $\shm\in\Derb_\coh(\sha)$ and let $a\in \Z$. The conditions below
are equivalent: 
\bnum
\item $H^a(\gr(\shm)) \simeq 0$,
\item $H^a(\shm) \simeq 0$ and $H^{a+1}(\shm)$ has no $\h$-torsion.
\enum
\end{proposition}

\subsubsection*{Cohomologically $\h$-complete sheaves on real manifolds}
Let now $X$ be a real analytic manifold.
Recall from~\cite{KS90} that the microsupport of $F\in\Derb(\Z_X)$ is
a closed involutive subset of the cotangent bundle $T^*X$ denoted by
$\SSi(F)$. 
The microsupport is additive on $\Derb(\Z_X)$
(cf~Definition~\ref{def:additivity}~(ii) below).
Considering the distinguished triangle
$F\to[\h] F\to\gr F\to[+1]$,
one gets the estimate
\begin{equation}\label{eq:SSCh}
\SSi(\gr(F)) \subset \SSi(F).
\end{equation}

Using Proposition~\ref{pro:conserv} and~\ref{pro:cohcodirim}, one easily
proves:
\begin{proposition}\label{pro:sscohco}
Let $F\in\Derb(\Z_X[\hbar])$ and assume that $F$ is cohomologically
$\h$-complete. Then 
\eq\label{eq:SSgr}
&&\SSi(F)=\SSi(\gr(F)).
\eneq
\end{proposition}

For $\cora$ a commutative unital Noetherian
ring, one denotes by $\mdrc[\cora_X]$ 
the full subcategory of $\md[\cora_X]$ consisting of  $\R$-constructible 
sheaves and by $\Derb_\Rc(\cora_X)$ the full triangulated subcategory of
$\Derb(\cora_X)$ consisting of objects with $\R$-constructible cohomology.
In this paper, we shall mainly be interested with the case where $\cora$ is
either $\C$ or the ring of formal power series in an 
indeterminate $\h$, that we denote by
\[
\coro \eqdot \C\forl.
\]

By Proposition~\ref{pro:cohco1} one has

\begin{proposition}\label{pro:RCcohco}
Let $F\in \Derb_\Rc(\coro_X)$. Then $F$ is cohomologically $\h$-complete.
\end{proposition}

\begin{corollary}\label{cor:conservativeRc}
The functor
$\grh\cl \Derb_\Rc(\coro_X)\to \Derb_\Rc(\C_X)$
is conservative.
\end{corollary}

\begin{corollary}\label{cor:SSF}
For $F\in\Derb_\Rc(\coro_X)$, one has the equality
\eqn
&&\SSi(\gr(F)) = \SSi(F).
\eneqn
\end{corollary}

\begin{proposition}\label{prop:supp_and_grad}
For $F\in\Derb_\Rc(\coro_X)$ and $i\in \Z$ one has $\supp H^i(F)
\subset \supp H^i(\gr F)$. In particular if $H^i(\gr F) =0$ then
$H^i(F) =0$.  
\end{proposition}
\begin{proof}
We apply Proposition~\ref{prop:cohom_and_grad} to $F_x$ for any
$x\in X$.
\end{proof}

\section{Formal extension}

Let $X$ be a topological space, or more generally a site,
and let $\shr_0$ be a sheaf of rings on $X$. In this section, we
let 
\[
\shr \seteq \shr_0\forl = \prod\limits_{n\ge0}\shr_0\h^n
\]
be the formal extension of $\shr_0$, whose sections on an open subset $U$
are formal series $r=\sum_{n=0}^\infty r_j\h^n$, with
$r_j\in\sect(U;\shr_0)$.
Consider the associated functor
\begin{align}\label{eq:subhbar}
(\scbul)^\h\cl\md[\shr_0] &\to \md[\shr],\\
\notag
\shn &\mapsto \shn\forl = \prolim[n](\shr_n\tens[\shr_0]\shn),
\end{align}
where $\shr_n \seteq \shr/\h^{n+1}\shr$ is regarded as an $(\shr,\shr_0)$-bimodule.
Since $\shr_n$ is free of finite rank over $\shr_0$, the functor $(\scbul)^\h$ is left exact.
We denote by $(\scbul)^\rhb$ its right derived functor.

\begin{proposition}\label{pro:Rh}
For $\shn\in\Derb(\shr_0)$ one has
\[
\shn^\rhb \simeq \rhom[\shr_0](\shr^\loc/\h\shr,\shn),
\]
where $\shr^\loc/\h\shr$ is regarded as an $(\shr_0,\shr)$-bimodule.
\end{proposition}

\begin{proof}
It is enough to prove that for $\shn\in\Mod(\shr_0)$  one has
\[
\shn^\h \simeq \hom[\shr_0](\shr^\loc/\h\shr,\shn).
\]
Let $\shr_n^* = \hom[\shr_0](\shr_n,\shr_0)$, regarded as an $(\shr_0,\shr)$-bimodule.
Then
\[
\shn^\h = \prolim[n](\shr_n\tens[\shr_0]\shn) \simeq \hom[\shr_0](\indlim[n]\shr_n^*,\shn).
\]
Since 
\[
\shr^\loc/\h\shr \simeq \indlim[n](\h^{-n}\shr/\h\shr),
\]
it is enough to prove that there is an isomorphism of $(\shr_0,\shr)$-bimodules
\[
\hom[\shr_0](\shr_n,\shr_0)\simeq\h^{-n}\shr/\h\shr.
\]
Recalling that $\shr_n = \shr/\h^{n+1}\shr$,
this follows from the pairing
\[
(\shr/\h^{n+1}\shr) \tens[\shr_0] (\h^{-n}\shr/\h\shr) \to \shr_0,\quad
f\otimes g \mapsto \operatorname{Res}_{\h=0}(fg\,d\h/\h).
\]
\end{proof}

Note that the isomorphism of $(\shr,\shr_0)$-bimodules
\[
\shr\simeq(\shr_0)^\h = \hom[\shr_0](\shr^\loc/\h\shr,\shr_0)
\]
induces a natural morphism
\begin{equation}\label{eq:RNtoNh}
\shr\ltens[\shr_0]\shn \to \shn^\rhb,\quad\text{for }\shn\in\Derb(\shr_0).
\end{equation}

\begin{proposition}\label{pro:Rhcomplete}
For $\shn\in\Derb(\shr_0)$, its formal extension $\shn^\rhb$ is cohomologically $\h$-complete.
\end{proposition}

\begin{proof}
The statement follows from $(\shr^\loc/\h\shr)\ltens[\shr]\shr^\loc \simeq 0$ and from the isomorphism
\[
\rhom[\shr](\shr^\loc,\shn^\rhb) \simeq \rhom[\shr_0]((\shr^\loc/\h\shr)\ltens[\shr]\shr^\loc,\shn).
\]
\end{proof}

\begin{lemma} \label{lem:Rh_hom}
Assume that $\shr_0$ is an $\shs_0$-algebra, for $\shs_0$ a commutative sheaf of rings,
and let $\shs = \shs_0\forl$.
For $\shm,\shn \in \Derb(\shr_0)$ we have an isomorphism in $\Derb(\shs)$
\eqn
&&\rhom[\shr_0](\shm,\shn)^\rhb \simeq \rhom[\shr_0](\shm,\shn^\rhb).
\eneqn
\end{lemma}

\begin{proof}
Noticing that $\shr^\loc/\h\shr \simeq \shr_0 \tens[\shs_0] (\shs^\loc/\h\shs)$ as $(\shr_0,\shs)$-bimodules, one has
\begin{align*}
\rhom[\shr_0](\shm,\shn)^\rhb 
&= \rhom[\shs_0] (\shs^\loc/\h\shs, \rhom[\shr_0](\shm,\shn)) \\
&\simeq \rhom[\shr_0] (\shr^\loc/\h\shr, \rhom[\shr_0](\shm,\shn)) \\
&\simeq \rhom[\shr_0](\shm, \rhom[\shr_0] (\shr^\loc/\h\shr,\shn))\\
&= \rhom[\shr_0](\shm,\shn^\rhb).
\end{align*}
\end{proof}

\begin{lemma}\label{lem:rhbroim}
Let $f\cl Y\to X$ be a morphism of sites, and assume that $(\opb f\shr_0)^\h \simeq \opb f \shr$.
Then the functors $\roim f$ and $(\scbul)^\rhb$ commute, that is, for
$\shp\in\Derb(\opb f\shr_0)$ we have
$(\roim f \shp)^\rhb\simeq\roim f(\shp^\rhb)$ in $\Derb(\shr)$.
\end{lemma}

\begin{proof}
One has the isomorphism
\begin{align*}
\roim f(\shp^\rhb)
&= \roim f \rhom[\opb f \shr_0] (\opb f(\shr^\loc/\h\shr),\shp) \\
&\simeq \rhom[\shr_0] (\shr^\loc/\h\shr,\roim f \shp) \\
&= \roim f(\shp^\rhb).
\end{align*}
\end{proof}

\begin{proposition}\label{pro:crit-h-acyclic}
Let $\sht$ be either a basis of open subsets of the site $X$ or, 
assuming that $X$ is a locally compact topological space, a basis of
compact subsets. Denote by
$J_\sht$ the full subcategory of $\md[\shr_0]$ consisting of
$\sht$-acyclic objects, {\em i.e.,} sheaves $\shn$ for which 
$H^k(S;\shn)= 0$ for all $k>0$ and all $S\in\sht$.
Then $J_\sht$ is injective with respect to the functor
$(\scbul)^\h$. In particular, for $\shn\in J_\sht$, we have
$\shn^\h\simeq \shn^\rhb$.
\end{proposition}

\begin{proof}
(i) Since injective sheaves are $\sht$-acyclic, $J_\sht$ is cogenerating. 

\noindent
(ii) Consider an exact sequence $0\to \shn'\to \shn\to \shn''\to 0$ in
$\md[\shr_0]$.
Clearly, if both $\shn'$ and $\shn$ belong to $J_\sht$, then so does $\shn''$.

\noindent
(iii) Consider an exact sequence as in (ii) and assume that $\shn'\in
J_\sht$. 
We have to prove that $0\to \shn^{\prime,\h}\to \shn^\h\to
\shn^{\prime\prime,\h}\to 0$ is exact. Since $(\scbul)^\h$ is left exact,
it is enough to prove that $\shn^\h\to
\shn^{\prime\prime,\h}$ is surjective.
Noticing that $\shn^\h \simeq \prod_\N\shn$ as $\shr_0$-modules,
it is enough to prove that $\prod_\N\shn \to
\prod_\N\shn^{\prime\prime}$ is surjective.

\noindent
(iii)-(a) Assume that $\sht$ is a basis of open subsets. 
Any open subset $U\subset X$ has a cover $\{U_i\}_{i\in I}$ by elements
$U_i\in\sht$. For any $i\in I$, the morphism $\shn(U_i)\to \shn''(U_i)$ is
surjective.  The result follows taking
the product over $\N$.

\noindent
(iii)-(b) Assume that $\sht$ is a basis of compact subsets. 
For any $K\in\sht$, the morphism $\shn(K)\to \shn''(K)$ is surjective.
Hence, there exists a basis $\shv$ of open subsets such that for any
$x\in X$ and any $V\ni x$ in $\shv$, there exists $V'\in\shv$ with
$x\in V'\subset V$ and the image of $\shn(V')\to \shn''(V')$
contains the image of $\shn''(V)$ in $\shn''(V')$. The result follows as in (iii)-(a)
taking the product over $\N$.
\end{proof}

\begin{corollary}\label{cor:hac}
The following sheaves are acyclic for the functor $(\scbul)^\h$:
\bnum
\item $\R$-constructible sheaves of $\C$-vector spaces on a real analytic manifold $X$  
{\rm (see~\cite[\S8.4]{KS90})},

\item coherent modules over the ring $\OO[X]$ of holomorphic functions 
on a complex analytic manifold $X$, 

\item coherent modules 
over the ring $\D[X]$ of linear differential operators on a complex
analytic  manifold $X$.
\enum
\end{corollary}

\begin{proof}
The statements follow by applying Proposition~\ref{pro:crit-h-acyclic}
for the following choices of $\sht$. 

\noindent
(i) Let $F$ be an $\R$-constructible sheaf. Then for any $x\in X$ one has 
$F_x\isofrom \rsect(U_x;F)$ for $U_x$ in a fundamental system of open 
neighborhoods of $x$. Take for $\sht$ the union of these fundamental
systems.

\noindent
(ii) Take for $\sht$ the
family of open Stein subsets. 

\noindent
(iii) Let $\shm$ be a coherent $\D[X]$-module. 
The problem being local, we may assume that $\shm$ is endowed with a
good filtration. Then take for $\sht$ the family of
compact Stein subsets.
\end{proof}

\begin{example}
Let $X = \R$, $\shr_0 = \C_X$, $Z = \{1/n\colon n=1,2,\dots\} \cup \{0\}$ and $U = X
\setminus Z$. 
One has the isomorphisms $(\coro)_X \simeq (\C_X)^\h \simeq (\C_X)^\rhb$
and 
$(\coro)_U \simeq (\C_U)^\h$. Considering the exact sequences
\begin{align*}
&0 \to (\coro)_U \to (\coro)_X \to (\coro)_Z \to 0, \\
&0 \to (\C_U)^\h \to (\C_X)^\h \to (\C_Z)^\h \to H^1(\C_U)^\rhb \to 0,
\end{align*}
we get $H^1(\C_U)^\rhb \simeq (\C_Z)^\h / (\coro)_Z$, whose stalk at the
origin does not vanish.
Hence $\C_U$ is not acyclic for the functor $(\scbul)^\h$.
\end{example}

Assume now that 
\[
\sha_0=\shr_0\quad\text{and}\quad\sha=\shr_0\forl
\]
satisfy either
Assumption~\ref{as:FDring} or Assumption~\ref{as:DQring} (where
condition (i) is clear) and that $\sha_0$ is syzygic.
Note that by Proposition~\ref{pro:crit-h-acyclic} one has
$\sha\simeq(\sha_0)^\rhb$.

\begin{proposition}\label{pro:nrh}
For $\shn\in\Derb_\coh(\sha_0)$:
\begin{itemize}
	\item [(i)] there is an isomorphism $\shn^\rhb\isoto\sha\ltens[\sha_0]\shn$ induced
	by \eqref{eq:RNtoNh},
	\item[(ii)] there is an isomorphism $\gr(\shn^\h)\simeq\shn$.
\end{itemize}
\end{proposition}

\begin{proof}
Since $\sha_0$ is syzygic, we may locally represent $\shn$
by a bounded complex $\shl^\scbul$ of free $\sha_0$-modules
of finite rank. Then (i) is obvious. As for (ii), both
complexes are isomorphic to the mapping cone of 
$\h\cl(\shl^\scbul)^\h \to (\shl^\scbul)^\h$.
\end{proof}

In particular, the functor $(\scbul)^\h$ is exact on $\Mod_\coh(\sha_0)$ and
preserves coherence. One thus get a functor
\[
(\scbul)^\rhb \colon \Derb_\coh(\sha_0) \to \Derb_\coh(\sha).
\]

\subsubsection*{The subanalytic site}

The subanalytic site associated to an analytic manifold $X$ has been
introduced and studied in~\cite[Chapter~7]{KS01} (see also \cite{Pr05}
for a detailed and systematic study as well as for complementary results).
Denote by $\Op_X$ the category of open subsets of $X$, the morphisms being the
inclusion morphisms, and by $\Op_\Xsa$ the full subcategory consisting of 
relatively compact subanalytic open subsets of $X$.
The site $\Xsa$ is the presite  $\Op_\Xsa$
endowed with the Grothendieck
topology for which the coverings are those admitting
a finite subcover.  One
calls $\Xsa$
the subanalytic site associated to $X$.
Denote by $\rho \cl X\to\Xsa$ the natural morphism of sites.
Recall that the inverse image functors $\opb\rho$, besides the
usual right adjoint given by the direct image functor $\oim\rho$,
admits a left adjoint denoted $\eim\rho$. Consider the diagram
\[
\xymatrix{
\Derb(\C_X) \ar@<.5ex>[r]^{\roim{\rho}} \ar[d]^{(\scbul)^\rhb} & 
\Derb(\C_{\Xsa}) \ar@<.5ex>[l]^{\opb\rho} \ar[d]_{(\scbul)^\rhb} \\
\Derb(\coro_X) \ar@<.5ex>[r]^{\roim{\rho}} & 
\Derb(\coro_{\Xsa}) \ar@<.5ex>[l]^{\opb\rho} .
}
\]

\begin{lemma}\label{lem:Rhandsa}
\bnum
\item
The functors $\opb{\rho}$ and $(\scbul)^\rhb$ commute, that is, for
$G\in\Derb(\C_\Xsa)$ we have
$(\opb{\rho}G)^\rhb\simeq\opb{\rho}(G^\rhb)$ in $\Derb(\coro_X)$.
\item
The functors $\roim{\rho}$ and $(\scbul)^\rhb$ commute, that is, for
$F\in\Derb(\C_X)$ we have
$(\roim{\rho}F)^\rhb\simeq\roim{\rho}(F^\rhb)$ in $\Derb(\coro_\Xsa)$.
\enum
\end{lemma} 

\begin{proof}
(i) Since it admits a 
left adjoint, the functor $\opb\rho$ commutes with projective limits. 
It follows that for $G\in\md[\C_\Xsa]$ one has an isomorphism
\[
\opb{\rho}(G^\h)\to(\opb{\rho}G)^\h.
\]
To conclude, it remains to show that $(\opb{\rho}(\scbul))^\rhb$
is the derived functor of $(\opb{\rho}(\scbul))^\h$. 
Recall that an object
$G$ of $\md[\C_\Xsa]$ is quasi-injective if the functor
$\Hom[{\C_\Xsa}](\scbul,G)$ is exact on the category
$\mdrc[\C_X]$.
By a result of \cite{Pr05}, if $G\in\md[\C_\Xsa]$ is quasi-injective,
then $\opb{\rho}G$ is soft. Hence, $\opb{\rho}G$ is injective for the
functor $(\scbul)^\h$ by Proposition~\ref{pro:crit-h-acyclic}.

\medskip
\noindent
(ii) By (i) we can apply Lemma~\ref{lem:rhbroim}.
\end{proof}

\section{$\D[]\forl$-modules and propagation}\label{section:Dh}

Let now $X$ be a complex analytic manifold of complex dimension $d_X$.
As usual, denote by $\C_X$ the constant sheaf with stalk $\C$, by
$\OO[X]$ the structure sheaf and by $\D[X]$ the ring of linear
differential operators on $X$. We will use the notations
\begin{align*}
\RD' &\cl \Derb(\C_X)^\rop \to \Derb(\C_X),& 
F &\mapsto\rhom[\C_X](F,\C_X), \\
\RDd&\cl \Derb_\coh(\D)^\rop \to \Derb_\coh(\D), &
\shm&\mapsto\rhom[\D](\shm,\D\tens[\sho_X]\Omega_X^{\otimes-1})\,[d_X],
\\
\Sol &\cl \Derb_\coh(\D)^{{\rm op}}\to\Derb(\C_X), &
\shm &\mapsto\rhom[{\D}](\shm,\OO),\\
\DR &\cl \Derb_\coh(\D)\to\Derb(\C_X), &
\shm &\mapsto \rhom[{\D}](\OO,\shm),
\end{align*}
where $\Omega_X$ denotes the line bundle of holomorphic forms of
maximal degree and $\Omega_X^{\otimes-1}$ the dual bundle.

As shown in Corollary~\ref{cor:hac}, the sheaves $\C_X$, $\OO[X]$ and
$\D[X]$ are all acyclic for the functor $(\scbul)^\h$.
We will be interested in the formal extensions
\[
\coro_X = \C_X\forl, \quad
\Oh = \OO[X]\forl, \quad
\Dh = \D[X]\forl.
\]
In the sequel, we shall treat left $\Dh$-modules, but all results
apply to right modules since the categories $\md[{\Dh}]$ and
$\md[{\D^{\h,{\rm op}}}]$ are equivalent.

\begin{proposition}
The $\coro$-algebras $\Dh$ and $\D^{\h,{\rm op}}$ satisfy
Assumptions~\ref{as:FDring}.
\end{proposition}

\begin{proof}
Assumption~\ref{as:FDring} hold for $\sha=\Dh$, $\shao=\D$,
$\mdgd[\shao\vert_U]$ 
the category of good $\D[U]$-modules (see \cite{Ka03}) and for $\BB$
the family of Stein compact subsets of $X$.
\end{proof}

In particular, by Theorem~\ref{th:formalfini2} one has that $\Dh$ is
right and left
Noetherian (and thus coherent). Moreover, by Theorem~\ref{th:hddim} any
object of $\Derb_\coh(\Dh)$ can be locally represented by a bounded
complex of free $\Dh$-modules of finite rank.

We will use the notations
\begin{align*}
\RD'_\h &\cl \Derb(\coro_X)^\rop \to \Derb(\coro_X),&
F &\mapsto\rhom[\coro_X](F,\coro_X),\\
\RDdh&\cl \Derb_\coh(\Dh)^\rop \to \Derb_\coh(\Dh), &
\shm&\mapsto\rhom[\Dh](\shm,\Dh\tens[\sho_X]\Omega_X^{\otimes-1})\,[d_X]
, \\
\Solh&\cl \Derb_\coh(\Dh)^{{\rm op}}\to\Derb(\coro), &
\shm &\mapsto\rhom[{\Dh}](\shm,\Oh),\\
\DRh&\cl \Derb_\coh(\Dh)\to\Derb(\coro), &
\shm &\mapsto \rhom[{\Dh}](\Oh,\shm).
\end{align*}

By Proposition~\ref{pro:nrh} and Lemma~\ref{lem:Rh_hom}, for $\shn\in\Derb_\coh(\D)$
one has
\begin{align}
\shn^\rhb &\simeq \Dh\ltens[\D]\shn, \\
\label{eq:grhn}
\gr(\shn^\rhb) &\simeq \shn, \\
\Solh(\shn^\rhb) &\simeq \Sol(\shn)^\rhb.
\end{align}

\begin{definition}\label{def:htor}
For $\shm \in \md[\Dh]$, denote by $\shm_\htor$ its submodule
consisting of
sections locally annihilated by some power of $\h$
and set $\shm_\htf = \shm/\shm_\htor$.
We say that $\shm\in\md[\Dh]$ is an $\h$-torsion module if $\shm_\htor\isoto\shm$ 
and that $\shm$ has no $\h$-torsion (or is $\h$-torsion
free) if $\shm\isoto\shm_\htf$.
\end{definition}

Denote by ${}_n\shm$ the kernel of 
$\h^{n+1}\colon \shm \to \shm$. Then $\shm_\htor$ is the sheaf
associated
with the increasing union of the ${}_n\shm$'s.
Hence, if $\shm$ is coherent, the increasing family $\{{}_n\shm\}_n$ is
locally stationary and $\shm_\htor$ as well as 
$\shm_\htf$ are coherent.

\subsubsection*{Characteristic variety}

Recall the following definition

\begin{definition}\label{def:additivity}
(i) For $\shc$ an abelian category, a function $c\colon \Ob(\shc)\to\mathrm{Set}$ 
is called additive if $c(M) = c(M') \cup c(M'')$ for any short exact sequence 
$0\to M' \to M \to M''\to 0$.

\noindent(ii)
For $\sht$ a triangulated category, a function $c\colon \Ob(\sht)\to\mathrm{Set}$ 
is called additive if $c(M) = c(M[1])$ and $c(M) \subset c(M') \cup c(M'')$ for any distinguished 
triangle $M' \to M \to M''\to[+1]$. 
\end{definition}

Note that an additive function $c$ on $\shc$ naturally extend to the derived category $\Der(\shc)$ 
by setting $c(M) = \bigcup_i c(H^i(M))$.

\medskip
For $\shn$ a coherent $\D[X]$-module, denote by $\chv(\shn)$ its
characteristic variety, a closed involutive subvariety of the cotangent
bundle $T^*X$. 
The characteristic variety is additive on $\Mod_\coh(\D[X])$.
For $\shn\in\Derb_\coh(\D[X])$ one sets
$\chv(\shn) = \bigcup_i \chv(H^i(\shn))$.

\begin{definition}
The characteristic variety of $\shm\in\Derb_\coh(\Dh)$ is defined by 
\eqn 
&&\chv_\h(\shm) = \chv(\grh(\shm)). 
\eneqn
\end{definition}

To $\shm \in \mdcoh[\Dh]$ one associates the coherent $\D$-modules
\begin{align}
{}_0\shm &= \ker(\h\colon \shm \to \shm) = H^{-1}(\gr\shm), \\
\shm_0 &= \coker(\h\colon \shm \to \shm) = H^0(\gr\shm).
\end{align}

\begin{lemma} \label{lem:tormod_chvH-1=chvH0}
For $\shm \in \mdcoh[\Dh]$ an $\h$-torsion module, one has 
\[
\chv_\h(\shm) = \chv(\shm_0) = \chv({}_0\shm).
\]
\end{lemma}

\begin{proof}
By definition, $\chv_\h(\shm) = \chv(\shm_0) \cup \chv({}_0\shm)$.
It is thus enough to prove the equality $\chv(\shm_0) = \chv({}_0\shm)$.

Since the statement is local we may assume that $\h^N\shm = 0$ 
for some $N\in\N$. We proceed by induction on $N$. 

For $N=1$ we have $\shm \simeq \shm_0 \simeq {}_0\shm$, and the
statement is obvious.

Assume that the statement has been proved for $N-1$. The
short exact sequence
\begin{equation}
\label{eq:ex_seq_tor}
0 \to \h\shm \to \shm \to \shm_0 \to 0
\end{equation}
induces the distinguished triangle
\[
\gr\h\shm \to \gr\shm \to \gr\shm_0 \to[+1].
\]
Noticing that $\shm_0 \simeq (\shm_0)_0 \simeq {}_0(\shm_0)$, the 
associated long exact cohomology sequence gives
\[
0 \to {}_0(\h\shm) \to {}_0\shm \to \shm_0 \to (\h\shm)_0 \to 0.
\]
By inductive hypothesis we have $\chv({}_0(\h\shm)) = \chv((\h\shm)_0)$, and we deduce
$\chv(\shm_0) = \chv(\shm_0)$ by additivity of $\chv(\scbul)$.
\end{proof}

\begin{proposition}\label{prop:add_var_car}
(i) For $\shm\in\Mod_\coh(\Dh)$ one has
$$
\chv_\h(\shm) = \chv(\shm_0).
$$
(ii) The characteristic variety $\chv_\h(\scbul)$ is additive both on
$\mdcoh[\Dh]$ and on $\Derb(\Dh)$.
\end{proposition}

\begin{proof}
(i) As $\chv(\gr\shm) = \chv(\shm_0) \cup \chv({}_0\shm)$,
it is enough to prove the inclusion
\begin{equation}
\label{eq:0MM0}
\chv({}_0\shm) \subset \chv(\shm_0).
\end{equation}
Consider the short exact sequence 
$0 \to \shm_\htor \to \shm \to \shm_\htf \to 0$.
Since $\shm_\htf$ has no $\h$-torsion, ${}_0(\shm_\htf)= 0$. 
The associated long exact cohomology sequence thus gives
$$
{}_0(\shm_\htor) \simeq {}_0\shm, \quad
0 \to (\shm_\htor)_0 \to \shm_0 \to (\shm_\htf)_0 \to 0.
$$
We deduce
$$
\chv({}_0\shm) = \chv({}_0(\shm_\htor))
= \chv((\shm_\htor)_0)  \subset \chv(\shm_0),
$$
where the second equality follows from Lemma~\ref{lem:tormod_chvH-1=chvH0}.

\smallskip\noindent
(ii) It is enough to prove the additivity on $\mdcoh[\Dh]$, i.e.\ the equality
\[
\chv_\h(\shm) = \chv_\h(\shm') \cup \chv_\h(\shm'')
\]
for $0 \to \shm' \to \shm \to \shm'' \to 0$ a short
exact sequence of coherent $\Dh$-modules.

The associated distinguished
triangle $\gr\shm' \to \gr\shm \to \gr\shm'' \to[+1]$
induces the long exact cohomology sequence
$$
{}_0(\shm'') \to (\shm')_0 \to \shm_0 
\to (\shm'')_0  \to 0.
$$
By additivity of $\chv(\scdot)$,
the exactness of this sequence at the first, second and
third term from the right, respectively, gives:
\begin{align*}
\chv_\h(\shm'') &\subset \chv_\h(\shm), \\
\chv_\h(\shm) &\subset \chv_\h(\shm') \cup \chv_\h(\shm''), \\
\chv_\h(\shm') &\subset \chv({}_0(\shm'')) \cup \chv_\h(\shm).
\end{align*}
Finally, note that $\chv({}_0(\shm'')) \subset \chv_\h(\shm'') \subset \chv_\h(\shm)$.
\end{proof}

\begin{remark}
In view of Proposition~\ref{prop:add_var_car}~(i), in order to define 
the characteristic variety of a coherent $\Dh$-module $\shm$ one
could avoid derived categories considering $\chv(\shm_0)$ instead of
$\chv(\gr\shm)$. 
It is then natural to ask if these definitions are still compatible
for $\shm\in\Derb_\coh(\Dh)$, i.e.~to ask
if the following equality holds
\[
\bigcup\nolimits_i \chv(H^i(\gr\shm)) = \bigcup\nolimits_i \chv((H^i\shm)_0).
\]
Let us prove it.
By additivity of $\chv(\scbul)$, the short exact sequence
\[
0 \to (H^i\shm)_0 \to H^i(\gr\shm) \to {}_0(H^{i+1}\shm) \to 0
\]
from \cite[Lemma~1.4.2]{KS08} induces the estimates
\begin{align*}
\chv((H^i\shm)_0) &\subset \chv(H^i(\gr\shm)), \\
\chv(H^i(\gr\shm)) &= \chv((H^i\shm)_0) \cup \chv({}_0(H^{i+1}\shm)).
\end{align*}
One concludes by noticing that \eqref{eq:0MM0} gives
\[
\chv({}_0(H^{i+1}\shm)) \subset \chv((H^{i+1}\shm)_0).
\]
\end{remark}

\begin{proposition}\label{pro:tormod_Dcoh}
Let $\shm \in \md[\Dh]$ be an $\h$-torsion module. Then $\shm$ is coherent
as a $\Dh$-module if and only if it is coherent as a $\D$-module, and in this case 
one has $\chv_\h(\shm) = \chv(\shm)$.
\end{proposition}

\begin{proof}
As in the proof of Lemma~\ref{lem:tormod_chvH-1=chvH0} we assume
that $\h^N\shm=0$ for some $N\in\N$.  Since
coherence is preserved by extension and since the characteristic
varieties of $\Dh$-modules and $\D$-modules are additive, we can argue by
induction on $N$ using the exact sequence~\eqref{eq:ex_seq_tor}.
We are thus reduced to the case $N=1$, where $\shm = \shm_0$
and the statement becomes obvious.
\end{proof}

It follows from \eqref{eq:grhn} that

\begin{proposition} \label{pro:formod_Dcoh}
For $\shn \in \Derb_\coh(\D)$ one has $\chv_\h(\shn^\h) = \chv(\shn)$.
\end{proposition}

\subsubsection*{Holonomic modules}

Recall that a coherent $\D$-module (or an object of the derived category) is called
holonomic if its characteristic variety is isotropic. We refer e.g.~to \cite[Chapter~5]{Ka03} for the 
notion of regularity.

\begin{definition}
We say that $\shm\in\Derb_\coh(\Dh)$ is holonomic, or regular holonomic,
if so is $\gr(\shm)$.
We denote by $\Derb_\hol(\Dh)$ the full triangulated subcategory of
$\Derb_\coh(\Dh)$ of holonomic objects and by $\Derb_\rh(\Dh)$ the
full triangulated subcategory of regular holonomic objects.
\end{definition}

Note that a coherent $\Dh$-module is holonomic
if and only if its characteristic variety is isotropic.

\begin{example}
Let $\shn$ be a regular holonomic $\D$-module. Then

\noindent
(i) $\shn$ itself, considered as a $\Dh$-module, is regular holonomic, as follows from the isomorphism $\gr\shn\simeq\shn\oplus\shn\,[1]$;

\noindent
(ii) $\shn^\h$ is a regular holonomic $\Dh$-module, as follows from the isomorphism $\gr\shn^\h\simeq\shn$.
In particular, $\Oh$ is regular holonomic.
\end{example}

\subsubsection*{Propagation}

Denote by $\Derb_\Cc(\coro_X)$ the full triangulated subcategory of
$\Derb(\coro_X)$ consisting of objects with $\C$-constructible
cohomology over the ring $\coro$.

\begin{theorem}\label{thm:SSDh}
Let $\shm,\shn\in\Derb_\coh(\Dh)$. Then 
\[
\SSi\bl\rhom[\Dh](\shm,\shn)\br = \SSi\bl
\rhom[\D](\gr(\shm),\gr(\shn))\br.
\]
If moreover $\shm$ and $\shn$ are holonomic, then
$\rhom[\Dh](\shm,\shn)$
is an object of $\Derb_\Cc(\coro_X)$.
\end{theorem}

\begin{proof}
Set $F = \rhom[\Dh](\shm,\shn)$. Then $F$ is cohomologically 
$\h$-complete by Theorem~\ref{th:formalfini1} and
Proposition~\ref{pro:homcc}. 
Hence $\SSi(F)=\SSi(\gr(F))$
by Proposition~\ref{pro:sscohco}. Moreover, the finiteness of the
stalks $\gr(F_x)$ over $\C$ implies the finiteness of $F_x$ over
$\coro$ by Theorem~\ref{th:formalfini2} applied with $X=\rmptt$ and
$\sha=\coro$.
\end{proof}

Applying Theorem~\ref{thm:SSDh}, and \cite[Theorem~11.3.3]{KS90}, we
get:
\begin{corollary}\label{cor:SSsol}
Let $\shm\in\Derb_\coh(\Dh)$. Then 
\[
\SSi(\Solh(\shm)) = \SSi(\DRh(\shm)) = \chv_\h(\shm). 
\]
If moreover $\shm$ is holonomic, then $\Solh(\shm)$ and $\DRh(\shm)$
belong to $\Derb_\Cc(\coro_X)$.
\end{corollary}

\begin{theorem}\label{thm:SSDh3}
Let $\shm\in\Derb_\hol(\Dh)$.  Then there is a natural isomorphism in
$\Derb_\Cc(\coro_X)$
\begin{equation}\label{eq:dualmorph}
\Solh(\shm)\simeq \RD'_\h(\DRh(\shm)).
\end{equation}
\end{theorem}

\begin{proof}
The natural $\coro$-linear morphism
\eqn
&&\hspace{-1cm}\rhom[{\Dh}](\Oh,\shm) \tens[\coro_X]
\rhom[{\Dh}](\shm,\Oh)\\
&&\hspace{6cm}\to \rhom[{\Dh}](\Oh,\Oh)\simeq \coro_X
\eneqn
induces the morphism in $\Derb_\Cc(\coro_X)$
\begin{equation}\label{eq:dualmor1}
\alpha\cl\rhom[{\Dh}](\shm,\Oh)\to\RD'_\h(\rhom[{\Dh}](\Oh,\shm)).
\end{equation}
(Note that, choosing $\shm=\Dh$, this morphism defines the morphism
$\Oh\to\RD'_\h(\Omegah[X][d_X])$.)
The morphism~\eqref{eq:dualmor1} induces an isomorphism
\eqn
&&\gr(\alpha)\cl\rhom[\D](\gr(\shm),\O) 
\to \RD'(\rhom[\D](\O,\gr(\shm))).
\eneqn
It is thus an isomorphism by Corollary~\ref{cor:conservativeRc}.
\end{proof}

\section{Formal extension of tempered functions}

Let us start by reviewing after \cite[Chapter~7]{KS01} the
construction of the sheaves of tempered distributions
and of $C^\infty$-functions with temperate growth on the subanalytic
site.

\medskip
Let $X$ be a real analytic manifold $X$. 
One says that a function $f\in \shc^{\infty}_X(U)$ has
{\it  polynomial growth} at $p\in X$
if, for a local coordinate system
$(x_1,\dots,x_n)$ around $p$, there exist
a sufficiently small compact neighborhood $K$ of $p$
and a positive integer $N$
such that
\eq
&\sup_{x\in K\cap U}\big(\dist(x,K\setminus U)\big)^N\vert f(x)\vert
<\infty\,.&
\eneq
One says that $f$ is {\it tempered} at $p$ if all its derivatives are
of polynomial growth at $p$.
One says that $f$ is tempered if it is tempered at any point of $X$.
One denotes by $\Ct_X(U)$ the $\C$-vector subspace of
$\shc^{\infty}(U)$ consisting of tempered functions.
It then follows from a theorem of Lojaciewicz that $U\mapsto\Ct_X(U)$
($U\in\Op_\Xsa$) is a sheaf on $\Xsa$. We denote it by $\Ct_\Xsa$ or
simply $\Ct_X$
if there is no risk of confusion.

\begin{lemma}\label{lem:Ctrhb}
One has
$H^j(U;\Ct_X)= 0$ for any $U\in\Op_\Xsa$ and any $j>0$.
\end{lemma}
This result is well-known (see \cite[Chapter~1]{KS96}), but we recall
its proof for the
reader's convenience. 
\begin{proof}
Consider the full subcategory $\shj$ of $\md[\C_\Xsa]$ consisting
of sheaves $F$ such that  for any pair
$U,V\in\Op_\Xsa$, the Mayer-Vietoris sequence
\[
0\to F(U\cup V)\to F(U)\oplus F(V)\to F(U\cap V)\to 0
\]
is exact. Let us check that this category is injective with
respect to the functor $\sect(U;\scbul)$. The only non obvious
fact is that if $0\to F'\to F\to F''\to 0$ is an exact sequence and that
$F'$ belongs to $\shj$, then $F(U)\to F''(U)$ is surjective.  
Let $t\in F''(U)$. There exist a finite covering $U=\bigcup_{i\in I}
U_i$ and 
$s_i\in F(U_i)$ whose image in $F''(U_i)$ is $t\vert_{U_i}$. Then the
proof goes by induction on the cardinal of $I$ using the property of
$F'$ and standard arguments. To conclude, note that $\Ct_X$ belongs to
$\shj$ thanks to  Lojaciewicz's result (see \cite{Ma}).
\end{proof}

Let $\shd b_X$ be the sheaf of distributions on $X$.
For $U\in\Op_\Xsa$,
denote by $\Dbt_X(U)$ the space of tempered distributions on $U$, 
defined by the exact sequence
\[
0\to\sect_{X\setminus U}(X;\shd b_X)\to\sect(X;\shd b_X)\to\shd
b_X^t(U)\to 0.
\]
Again, it follows from a theorem of Lojaciewicz that $U\mapsto\Dbt(U)$
is a sheaf on $\Xsa$. We denote it by $\Dbt_\Xsa$ or simply $\Dbt_X$
if there is no risk of confusion.
The sheaf $\Dbt_X$ is quasi-injective,
that is, the functor $\hom[\C_\Xsa](\scbul,\Dbt_X)$ is exact in the
category $\mdrc[\C_X]$.
Moreover, for $U\in\Op_\Xsa$, 
$\hom[{\C_\Xsa}](\C_U,\Dbt_X)$ is also quasi-injective and 
$\rhom[\C_\Xsa](\C_U,\Dbt_X)$ is concentrated in degree $0$. Note that
the sheaf
\[
\sect_{[U]}\shd b_X\eqdot\opb{\rho}\hom[\C_\Xsa](\C_U,\Dbt_X)
\]
is a $\shc^\infty_X$-module, so that in particular $\rsect(V;\sect_{[U]}\shd b_X)$ 
is concentrated in degree $0$ for $V\subset X$ an open subset.

\subsubsection*{Formal extensions}

By Proposition~\ref{pro:crit-h-acyclic} the sheaves 
$\Cth_X$, $\Dbth_X$ and $\sect_{[U]}\shd b_X$
are acyclic for the functor $(\scbul)^\h$.
We set
\[
\Cth_X\eqdot(\Ct_X)^\h, \qquad
\Dbth_X\eqdot(\Dbt_X)^\h, \qquad
\sect_{[U]}\shd b_X^\h\eqdot (\sect_{[U]}\shd b_X)^\h.
\]
Note that, by Lemmas~\ref{lem:Rh_hom} and \ref{lem:Rhandsa},
\[
\sect_{[U]}\shd b_X^\h \simeq \opb{\rho}\hom[\C_\Xsa](\C_U,\Dbth_X).
\]
By Proposition~\ref{pro:Rhcomplete} we get:

\begin{proposition}\label{pro:ctdbtcc}
The sheaves $\Cth_X$, $\Dbth_X$ and $\sect_{[U]}\shd b_X^\h$ are cohomologically $\h$-complete.
\end{proposition}

Now assume $X$ is a complex manifold. Denote by $\overline{X}$ the
complex conjugate manifold and by $X^\R$ the underlying real analytic
manifold, identified with the diagonal of $X\times \overline{X}$.  One
defines the sheaf (in fact, an object of the derived category)
of tempered holomorphic functions by
\[
\Ot \eqdot 
\rhom[\eim{\rho}\shd_{\overline{X}}](\eim{\rho}\sho_{\overline{X}},\Ct_X)
\isoto\rhom[\eim{\rho}\shd_{\overline{X}}](\eim{\rho}\sho_{\overline{X}},\Dbt_X).
\]
Here and in the sequel, we write $\Ct_X$ and $\Dbt_X$ instead of
$\Ct_{X^\R}$ and $\Dbt_{X^\R}$, respectively.
We set
\[
\Oth \eqdot (\Ot)^\rhb,
\]
a cohomologically $\h$-complete object of $\Derb(\coro_{\Xsa})$.
By Lemma~\ref{lem:Rh_hom},
\eqn
\Oth&\simeq& 
\rhom[\eim{\rho}\shd_{\overline{X}}](\eim{\rho}\sho_{\overline{X}},\Cth_X)
\isoto
\rhom[\eim{\rho}\shd_{\overline{X}}](\eim{\rho}\sho_{\overline{X}},\Dbth_X).
\eneqn
Note that $\gr(\Oth)\simeq\Ot$ in $\Derb(\C_{\Xsa})$.

\section{Riemann-Hilbert correspondence}

Let $X$ be a complex analytic manifold.
Consider the functors
\begin{align*}
\TH(\scbul) &\cl \Derb_\Cc(\C_X) \to \Derb_\rh(\D)^\rop, &
F & \mapsto \opb{\rho}\rhom[\C_{\Xsa}](\oim{\rho}F,\Ot),\\
\THh(\scbul) &\cl \Derb_\Cc(\coro_X) \to \Derb(\Dh)^\rop, &
F & \mapsto \opb{\rho}\rhom[\coro_{\Xsa}](\oim{\rho}F,\Oth).
\end{align*}

The classical Riemann-Hilbert correspondence of Kashiwara~\cite{Ka84}
states
that the functors $\Sol$ and $\TH$
are equivalences of categories
between $\Derb_\Cc(\C_X)$ and $\Derb_\rh(\D)^{{\rm op}}$ quasi-inverse
to each other.
In order to obtain a similar statement for $\C_X$ and $\D$ 
replaced with $\coro_X$ and $\Dh$, respectively, we start by
establishing some lemmas.

\begin{lemma}\label{le:Dhff}
Let $\shm,\shn\in \Derb_\hol(\Dh)$. 
The natural morphism in $\Derb_\Cc(\coro_X)$
\[
\rhom[{\Dh}](\shm,\shn)\to \rhom[\coro_X](\Solh(\shn),\Solh(\shm))
\]
is an isomorphism.
\end{lemma}
\begin{proof}
Applying the functor $\gr$ to this morphism, we get an isomorphism by
the classical Riemann-Hilbert correspondence.
Then the result follows from Corollary~\ref{cor:conservativeRc} and
Theorem~\ref{thm:SSDh}.
\end{proof}

Note that there is an isomorphism in $\Derb(\D)$
\begin{equation}\label{eq:grRHh}
\gr(\THh(F))\simeq\TH(\gr(F)).
\end{equation}

\begin{lemma}\label{le:RHrhh}
The functor $\THh$ induces a functor
\begin{equation}\label{eq:RHrhh}
\THh\cl \Derb_\Cc(\coro_X)\to \Derb_\rh(\Dh)^\rop.
\end{equation}
\end{lemma}
\begin{proof}
Let $F\in\Derb_\Cc(\coro_X)$. By \eqref{eq:grRHh}
and the classical Riemann-Hilbert correspondence we know that
$\gr(\THh(F))$ is regular holonomic, and in particular coherent.
It is thus left to prove that $\THh(F)$ is coherent.
Note that our problem is of local nature.

We use
the Dolbeault resolution of $\Oth$ with coefficients in $\Dbth_X$ and
we choose a resolution of $F$ as given in
Proposition~\ref{pro:resol1}~(i).  We find that $\THh(F)$ is isomorphic
to a
bounded complex $\shm^\scbul$, where the $\shm^i$ are locally finite
sums of sheaves of the type $\sect_{[U]}\Dbth$ with $U\in\Op_\Xsa$. It
follows from Proposition~\ref{pro:ctdbtcc} that $\THh(F)$ is cohomologically
$\h$-complete, and we conclude by Theorem~\ref{th:formalfini2} with $\sha=\Dh$.
\end{proof}

\begin{lemma}\label{le:compl_sol}
We have
$\rhom[{\eim{\rho}\Dh}](\eim{\rho}\Oh,\Oth)\simeq\coro_{\Xsa}$.
\end{lemma}

\begin{proof}
This isomorphism is given by the sequence
\eqn
\rhom[{\eim{\rho}\Dh}](\eim{\rho}\Oh,\Oth)
&\simeq&\rhom[{\eim{\rho}\D}](\eim{\rho}\O,\Oth)\\
&\simeq&\rhom[\eim{\rho}\D](\eim{\rho}\O,\Ot)^\rhb\\
&\simeq&(\oim{\rho}\rhom[\D](\O,\O))^\rhb\simeq (\C_\Xsa)^\rhb\simeq\coro_{\Xsa},
\eneqn
where the first isomorphism is an extension of scalars, the second one
is Lemma~\ref{lem:Rh_hom} and the third one is given by the adjunction
between $\eim{\rho}$ and $\opb{\rho}$.
\end{proof}

\begin{theorem}\label{th:DhRH}
The functors $\Solh$ and $\THh$ are equivalences of categories
between $\Derb_\Cc(\coro_X)$ and $\Derb_\rh(\Dh)^{\rm op}$ quasi-inverse
to each other.
\end{theorem}

\begin{proof}
In view of 
Lemma~\ref{le:Dhff}, we know that the functor $\Solh$ is fully faithful.
It is then enough to show that
$\Solh(\THh(F))\simeq F$ for $F\in\Derb_\Cc(\coro_X)$.
Since we already know by Lemma~\ref{le:RHrhh} that $\THh(F)$ is
holonomic, we may
use~\eqref{eq:dualmorph}.  We have the sequence of isomorphisms:
\allowdisplaybreaks
\[
\begin{split}
\oim\rho\rhom[{\Dh}]&(\Oh, \THh(F))
= 
\oim\rho\rhom[{\Dh}](\Oh,\opb{\rho}\rhom[\coro_{\Xsa}]
(\oim{\rho}F,\Oth) )\\
&\simeq
\rhom[{\rho_!\Dh}](\eim{\rho}\Oh, \rhom[\coro_{\Xsa}] (\oim{\rho}
F,\Oth))\\
&\simeq
\rhom[\coro_{\Xsa}](\oim{\rho}F,
\rhom[{\eim{\rho}\Dh}](\eim{\rho}\Oh,\Oth)) \\
&\simeq
\rhom[\coro_{\Xsa}](\oim{\rho}F, \coro_{\Xsa})
\simeq
\rhom[\coro_{\Xsa}](\oim{\rho}F, \oim{\rho}\coro_X) \\
&\simeq 
\oim\rho\RD'_\h F,
\end{split}
\]
where we have used the adjunction between $\eim{\rho}$ and
$\opb{\rho}$, the isomorphism of Lemma~\ref{le:compl_sol} and the
commutation of $\oim{\rho}$ with $\rhom$.  One concludes by recalling
the isomorphism of functors $\opb{\rho}\oim\rho\simeq\id$.
\end{proof}

\subsubsection*{$t$-structure}
Recall the definition of the middle perversity $t$-structure for
complex constructible sheaves.  
Let $\cora$ denote either the field $\C$ or the
ring $\coro$.
For $F\in \Derb_\Cc(\cora_X)$, we have 
$F \in \pDer[\leq 0]_\Cc(\cora_X)$ if and only if
\begin{equation}
\label{eq:perv_negatif}
\forall i \in \Z \qquad \dim \supp H^i(F) \leq d_X -i ,
\end{equation}
and $F \in \pDer[\geq 0]_\Cc(\cora_X)$ if and only if, for any locally
closed complex analytic subset $S\subset X$,
\begin{equation}\label{eq:perv_positif}
H^i_S(F) =0 \mbox{ for all } i < d_X-\dim(S).
\end{equation}
With the above convention, the de Rham functor
\[
\DR\cl\Derb_\hol(\D)\to\pDer_\Cc(\C_X)  
\]
is $t$-exact. 
\begin{theorem}\label{thm:pervh} 
The de Rham functor $\DRh\cl\Derb_\hol(\Dh)\to\pDer_\Cc(\coro_X)$ is
$t$-exact.
\end{theorem}
\begin{proof}
(i) Let $\shm \in \Der[\leq 0]_\hol(\Dh)$. Let us prove that
$\DRh\shm\in\pDer[\leq 0]_\Cc(\coro_X)$.  Since $\DRh\shm$ is
constructible, Proposition~\ref{prop:supp_and_grad} shows that it is
enough to check~\eqref{eq:perv_negatif} for 
$\gr(\DRh\shm)\simeq\DR(\gr\shm)$. In other words, it is enough  
to check that $\DR(\grh\shm)\in\pDer[\leq0]_\Cc(\C_X)$.
Since $\grh\shm \in\Der[\leq 0]_\hol(\D)$, this result follows from the
$t$-exactness of the functor $\DR$.

\medskip
\noindent
(ii) Let $\shm \in \Der[\geq 0]_\hol(\Dh)$. Let us prove that
$\DRh\shm\in\pDer[\geq 0]_\Cc(\coro_X)$.  We set $\shn =
(H^0\shm)_\htor$. We have a morphism $u\cl \shn \to \shm$ induced by
$H^0\shm \to \shm$ and we let $\shm'$ be the mapping cone of $u$. We
have a distinguished triangle
$$
\DRh \shn \to \DRh \shm \to \DRh\shm' \xto{+1}
$$
so that it is enough to show that $\DRh \shn$ and $\DRh\shm'$
belong to $\pDer[\geq 0]_\Cc(\coro_X)$.

\smallskip
\noindent
(a) By Proposition~\ref{prop:add_var_car} (ii) and
Proposition~\ref{pro:tormod_Dcoh}, $\shn$ is holonomic as a
$\D$-module. Hence $\DRh \shn \simeq \DR \shn$ is a perverse sheaf
(over $\C$) and satisfies~\eqref{eq:perv_positif}.
Since~\eqref{eq:perv_positif} does not depend on the coefficient ring,
$\DRh \shn \in \pDer[\geq 0]_\Cc(\coro_X)$.

\smallskip
\noindent
(b) We note that $H^0\shm' \simeq (H^0\shm)_\htf$. Hence by
Proposition~\ref{prop:cohom_and_grad}, 
$\gr\shm'\in\Der[\geq0]_\hol(\D)$ and 
$\DR(\gr\shm') \in\pDer[\geq0]_\Cc(\C_X)$, that is,
$\DR(\gr\shm')$ satisfies~\eqref{eq:perv_positif}.  Let $S\subset X$
be a locally closed complex subanalytic subset. We have
$$
\rsect_S(\DR(\gr\shm')) \simeq \gr (\rsect_S (\DRh\shm'))
$$
and it follows from Proposition~\ref{prop:supp_and_grad} that
$\DRh\shm'$ also satisfies~\eqref{eq:perv_positif} and thus belongs to
$\pDer[\geq 0]_\Cc(\coro_X)$.
\end{proof}

\section{Duality and $\h$-torsion}

The duality functors $\RDd$ on $\Der_\rh(\D)$ and $\RD'$ on
$\pDer_\Cc(\C_X)$ are $t$-exact. 
We will discuss here the finer $t$-structures needed in order to obtain
a similar result
when replacing $\C_X$ and $\D[X]$ by their formal extensions $\coro_X$
and $\Dh[X]$.

\medskip
Following~\cite[Chapter~I.2]{HRS96}, let us start by recalling some
facts related to torsion pairs and $t$-structures. We need in
particular Proposition~\ref{prop:ttts} below, which can also be found
in~\cite{J08}.

\begin{definition}\label{def:tt}
Let $\shc$ be an abelian category. A torsion pair on $\shc$ is a pair
$(\shc_{\text{tor}}, \shc_{\text{tf}})$ of full subcategories such that
\begin{enumerate}[(i)]
\item \label{eq:hom torsion theory}
for all objects $T$ in $\shc_{\text{tor}}$ and $F$ in $\shc_{\text{tf}}$, we have
$\Hom[\shc](T,F)= 0$,
\item \label{eq:ses torsion theory}
for any object $M$ in $\shc$, there are objects $M_{\text{tor}}$ in
$\shc_{\text{tor}}$ and  $M_{\text{tf}}$ in $\shc_{\text{tf}}$
and a short exact sequence
$0 \to M_{\text{tor}} \to M \to M_{\text{tf}} \to 0$.
\end{enumerate}
\end{definition}

\begin{proposition}\label{prop:ttts}
Let $\Der$ be a triangulated category endowed with a $t$-structure
$(\pDer[\leq 0], \pDer[\geq 0])$. Let us denote its
heart by $\shc$ and its cohomology functors by $\pH^i\cl\Der\to\shc$. 
Suppose that $\shc$ is endowed with a torsion pair
$(\shc_{\text{tor}},\shc_{\text{tf}})$.
Then we can define a new $t$-structure
$(\piDer[\leq 0], \piDer[\geq 0])$ on $\Der$ by setting:
\eqn
&&\piDer[\leq 0] = \{M \in \pDer[\leq 1]\cl\pH^1(M)\in\shc_{\text{tor}}\},\\
&&\piDer[\geq 0] = \{M \in \pDer[\geq 0]\cl\pH^0(M)\in\shc_{\text{tf}}\}.
\eneqn
\end{proposition}

With the notations of Definition~\ref{def:htor}, there is a natural
torsion pair attached to $\md[\Dh]$ given by the full subcategories
\begin{align*}
\md[\Dh]_\htor &= \{\shm\cl \shm_\htor\isoto\shm\}, \\
\md[\Dh]_\htf &= \{ \shm\cl \shm \isoto\shm_\htf\}.
\end{align*}

\begin{definition}
\banum
\item
We call the torsion pair on $\md[\Dh]$ defined above, the 
$\h$-torsion pair.
\item
We denote by 
$\bigl(\Der[\leq 0](\Dh),\Der[\geq 0](\Dh)\bigr)$ the natural
$t$-structure on $\Der(\Dh)$.
\item
We denote by 
$\bigl(\tDer[\leq 0](\Dh),\tDer[\geq 0](\Dh)\bigr)$ the $t$-structure 
on $\Derb(\Dh)$ associated via Proposition~\ref{prop:ttts} with 
the $\h$-torsion pair on $\md[\Dh]$.
\eanum
\end{definition}

Proposition~\ref{prop:cohom_and_grad} implies the following
equivalences for $\shm \in \Derb_\coh(\Dh)$:
\begin{align}
\label{eq:tstruc_pos}
\shm \in \tDer[\geq 0](\Dh) &\Longleftrightarrow 
\grh\shm\in\Der[\geq 0](\D),  \\
\label{eq:tstruc_neg}
\shm \in \Der[\leq 0](\Dh) &\Longleftrightarrow 
\grh\shm\in\Der[\leq 0](\D) .
\end{align}

\begin{proposition}
\label{prop:duality_torsion}
Let $\shm$ be a holonomic $\D^\h$-module.
\bnum
\item
If $\shm$ has no $\h$-torsion, then $\RDdh \shm$ is concentrated in
degree $0$ and has no $\h$-torsion.
\item
If $\shm$ is an $\h$-torsion module, then $\RDdh \shm$ is
concentrated in degree $1$ and is an $\h$-torsion module.
\enum
\end{proposition}

\begin{proof}
By~\eqref{eq:grF} we have $\gr(\RDdh \shm) \simeq \RDd(\gr \shm)$.
Since $\gr \shm$ is concentrated in degrees $0$ and $-1$, with
holonomic cohomology, $\RDd(\gr \shm)$ is concentrated in degrees
$0$ and $1$.  By Proposition~\ref{prop:cohom_and_grad}, $\RDdh \shm$
itself is concentrated in degrees $0$ and $1$ and $H^0(\RDdh \shm)$
has no $\h$-torsion.

\medskip
\noindent
(i) The short exact sequence
\eqn
&&  0 \to \shm \to[\h] \shm \to \shm/\h\shm \to 0
\eneqn
induces the long exact sequence
$$
\cdots \to H^1(\RDdh(\shm/\h\shm)) \to H^1(\RDdh\shm)
\to[\h] H^1(\RDdh\shm) \to 0.
$$
By Nakayama's lemma $H^1(\RDdh\shm)= 0$ as required.

\medskip
\noindent
(ii) Since $\shm$ is locally annihilated by some power of $\h$, the
cohomology groups $H^i(\RDdh \shm)$ also are $\h$-torsion modules.  As
$H^0(\RDdh \shm)$ has no $\h$-torsion, we get $H^0(\RDdh \shm) = 0$.
\end{proof}

\begin{theorem}\label{th:tDhdual}
The duality functor $\RDdh\cl \Derb_\hol(\Dh)^\op\to\tDer_\hol(\Dh)$ is
$t$-exact.
In other words, $\RDdh$ interchanges $\Der[\leq0]_\hol(\Dh)$ with
$\tDer[\geq0]_\hol(\Dh)$
and $\Der[\geq0]_\hol(\Dh)$ with $\tDer[\leq0]_\hol(\Dh)$.
\end{theorem}

\begin{proof}
(i) Let us first prove for $\shm \in \Derb_\hol(\Dh)$:
\eq\label{eq:tDhdual1}
&&\shm\in\Der[\leq0]_\hol(\Dh)\Longleftrightarrow
\RDdh(\shm)\in\tDer[\geq0]_\hol(\Dh).
\eneq
By~\eqref{eq:grF} we have $\gr(\RDdh \shm) \simeq \RDd(\gr \shm)$
and we know that the analog of~\eqref{eq:tDhdual1} holds true for
$\D$-modules:
\[
\shn \in \Der[\leq0]_\hol(\D)
\Longleftrightarrow \RDd(\shn) \in \Der[\geq0]_\hol(\D).
\]
Hence~\eqref{eq:tDhdual1} follows easily from~\eqref{eq:tstruc_pos}
and~\eqref{eq:tstruc_neg}.

\medskip
\noindent
(ii) We recall the general fact for a $t$-structure 
$(\Der,\Der[\leq 0], \Der[\geq 0])$ and $ A\in \Der$:
\begin{align*}
A\in\Der[\leq 0]&\Longleftrightarrow\forall B\in \Der[\geq 1]\;\Hom(A,B) =0,\\ 
A\in\Der[\geq 0]&\Longleftrightarrow\forall B\in \Der[\leq -1]\;\Hom(B,A) =0.
\end{align*}
Since $\RDdh$ is an involutive equivalence of categories we deduce
from~\eqref{eq:tDhdual1} the dual statement: 
\eqn
&&\shm\in\Der[\geq0]_\hol(\Dh)
\Longleftrightarrow\RDdh(\shm)\in\tDer[\leq0]_\hol(\Dh).
\eneqn
\end{proof}

\begin{remark}
The above result can be stated as follows in the language of
quasi-abelian categories of \cite{Sn99}. We will follow the same
notations as in \cite[Chapter~2]{Ka08}. The category
$\shc=\md[\Dh]_\htf$ is quasi-abelian. Hence its derived category has a
natural generalized $t$-structure
$(\Der[\leq s](\shc),\Der[>s-1](\shc))_{s\in\frac12\Z}$. Note that
$\Der[{[-1/2,0]}](\shc)$ is equivalent to $\md[\Dh]$, and that
$\Der[{[0,1/2]}](\shc)$ is equivalent to the heart of $\tDer(\Dh)$.
Then Theorem~\ref{th:tDhdual} states that the duality functor $\RDdh$ is
$t$-exact on $\Derb_\hol(\shc)$. 
\end{remark}

Consider the full subcategories of $\Perv(\coro_X)$
\begin{align*}
\Perv(\coro_X)_\htor &= 
\{ F\cl \text{locally }\h^N F=0\text{ for some }N\in\N
\}, \\
\Perv(\coro_X)_\htf &= 
\{ F\cl F \text{ has no non zero subobjects in }
\Perv(\coro_X)_\htor\}.
\end{align*}
\begin{lemma}
\bnum
\item
Let $F \in \Perv(\coro_X)$. Then the inductive system of sub-perverse
sheaves $\ker(\h^n\colon F\to F)$ is locally stationary. 
\item
The pair $\bigl(\Perv(\coro_X)_\htor,\Perv(\coro_X)_\htf\bigr)$ is a
torsion pair.
\enum
\end{lemma}
\begin{proof}
(i) Set $\shm = \RDdh\THh(F)$. By the Riemann-Hilbert correspondence,
one has
$\ker(\h^n\colon F\to F) \simeq \DRh(\ker(\h^n\colon \shm\to \shm))$. 
Since $\shm$ is coherent, the inductive system $\ker(\h^n\cl\shm\to
\shm)$ 
is locally stationary. Hence so is the system $\ker(\h^n\colon F\to F)$.

\noindent
(ii) By (i) it makes to define for $F\in\Perv(\coro_X)$:
\[
F_\htor = \bigcup_n \ker(\h^n\colon F\to F), \quad F_\htf = F /
F_\htor.
\]
It is easy to check that $F_\htor\in\Perv(\coro_X)_\htor$ and
$F_\htf\in\Perv(\coro_X)_\htf$. Then property (ii) in
Definition~\ref{def:tt} is clear. For property (i) let $u \cl F\to G$
be a morphism in $\Perv(\coro_X)$ with $F \in \Perv(\coro_X)_\htor$
and $G \in \Perv(\coro_X)_\htf$. Then $\im u$ also is in
$\Perv(\coro_X)_\htor$ and so it is zero by definition of
$\Perv(\coro_X)_\htf$.
\end{proof}

Denote by
$\bigl(\piDer[\leq0]_\Cc(\coro_X),\piDer[\geq 0]_\Cc(\coro_X)\bigr)$ 
the $t$-structure on $\Der_\Cc(\coro_X)$
induced by the perversity $t$-structure and
this torsion pair as in Proposition~\ref{prop:ttts}.
We also set ${}^\pi\Perv(\coro_X) = \piDer[\leq0]_\Cc(\coro_X)\cap
\piDer[\geq 0]_\Cc(\coro_X)$.

\begin{corollary}\label{cor:pervh}
There is a quasi-commutative diagram of $t$-exact functors
\eqn
&&\xymatrix{ 
\Derb_\hol(\Dh)^{\rm op} \ar[r]^{\DRh} \ar[d]^{\RDdh} &
\pDer_\Cc(\coro_X)^{\rm op} \ar[d]^{\RD'_\h} \\ 
\tDer_\hol(\Dh) \ar[r]^{\DRh}  & \piDer_\Cc(\coro_X) 
}\eneqn
where the duality functors are equivalences of categories and the de
Rham functors become equivalences when restricted to the
subcategories of regular objects.
\end{corollary}

\begin{example}
Let $X = \C$, $U = X \setminus \{0\}$ and denote by $j\colon
U\hookrightarrow X$ the embedding.
Let $L$ be the local system on $U$ with stalk $\coro$ and monodromy
$1+\h$. 
The sheaf $\roim j L \simeq \RD'_h (\eim j (\RD'_h L))$
is perverse for both
$t$-structures, as is the sheaf
$H^0(\roim j L) = \oim j L \simeq \eim j L$.
The sheaf $H^1(\roim j L) \simeq \C_{\{0\}}$ has $\h$-torsion.
From the distinguished triangle $\oim j L \to \roim j L \to \C_{\{0\}}[-1] \to[+1]$, one gets the short exact
sequences
\begin{align*}
0 \to \oim j L \to \roim j L \to \C_{\{0\}}[-1] \to 0 &\quad\text{in
}\Perv(\coro_X),\\
0 \to \C_{\{0\}}[-2] \to \oim j L \to \roim j L \to 0 &\quad\text{in
}{}^\pi\Perv(\coro_X).
\end{align*}
\end{example}

\section{$\D[]\Ls$-modules}\label{section:Dhl}

Denote by 
\[
\C^{\h, {\rm loc}} := \C\Ls = \C[\h^{-1},\h]\mspace{-1mu}]
\]
the field of Laurent series in
$\h$, that is the fraction field of $\coro$. 
Recall the exact functor
\begin{equation}\label{eq:subhbarloc}
(\scbul)^{\rm loc}\cl\md[\coro_X] \to \md[\C^{\h, {\rm loc}}_X],
\quad F\mapsto\C^{\h, {\rm loc}}\tens[\coro]F,
\end{equation}
and note that	by \cite[Proposition~5.4.14]{KS90} one has the estimate
\begin{equation}\label{eq:SSloc}
\SSi(F^\loc)\subset\SSi(F).
\end{equation}
For $G\in\Derb(\C_X)$, we write $G^{\h, {\rm loc}}$ instead of
$(G^\h)^{\rm loc}$.
We will consider in particular 
\[
\Ohl = \O\Ls, \qquad \Dhl = \D\Ls.
\]

\begin{lemma}\label{le:pseudoco}
Let $\shm$ be a coherent $\Dhl$-module. Then $\shm$ is pseudo-coherent
over $\Dh$. In other word, if $\shl\subset\shm$ is a finitely
generated $\Dh$-module, then $\shl$ is $\Dh$-coherent. 
\end{lemma}

\begin{proof}
The proof follows from \cite[Appendix.~A1]{Ka03}.
\end{proof}

\begin{definition}
A lattice $\shl$ of a coherent $\Dhl[X]$-module $\shm$ is a coherent 
$\Dh[X]$-submodule of $\shm$ which generates it.
\end{definition}

Since $\shm$ has no $\h$-torsion, any of its lattices has no
$\h$-torsion. In particular, one has $\shm \simeq \shl^\loc$
and $\gr\shl \simeq \shl_0=\shl/\h\shl$.

It follows from Lemma~\ref{le:pseudoco} that lattices locally exist: for
a  local
system of generators $(m_1,\dots,m_N)$ of $\shm$, define $\shl$ as the
$\Dh$-submodule with the same generators.

\begin{lemma}\label{lem:additivity}
Let $0 \to \shm' \to \shm \to \shm'' \to 0$ be an exact sequence of
coherent $\Dhl[X]$-modules.  Locally there exist lattices $\shl'$,
$\shl$, $\shl''$ of $\shm'$, $\shm$, $\shm''$, respectively, 
inducing an exact sequence of $\Dh[X]$-modules
\[
0 \to \shl' \to \shl \to \shl'' \to 0.
\]
\end{lemma}

\begin{proof}
Let $\shl$ be a lattice of $\shm$ and let $\shl''$ be its
image in $\shm''$. We set $\shl'\eqdot\shl\cap\shm'$.  These
sub-$\Dh[X]$-modules give rise to an exact sequence.

Since $\shl''$ is of finite type over $\Dh[X]$, it is a lattice
of $\shm''$. Let us show that $\shl'$ is a lattice of $\shm'$.
Being the kernel of a morphism $\shl\to\shl''$ between coherent
$\Dh[X]$-modules, $\shl'$ is coherent. To show that $\shl'$
generates $\shm'$, note that any $m'\in\shm'\subset\shm$ may be
written as $m' = \h^{-N} m$ for some $N\geq 0$ and $m\in\shl$.
Hence $m = \h^N m' \in \shm'\cap \shl = \shl'$.
\end{proof}

For an abelian category $\shc$, we denote by $\K(\shc)$ its
Grothendieck group. 
For an object $M$ of $\shc$, we denote by $[M]$ its class in $\K(\shc)$.
We let $\shk(\shd_X)$ be the sheaf on $X$ associated to the presheaf
\[
U \mapsto \K(\mdc[\shd_X|_U]) .
\]
We define $\shk(\Dhl[X])$ in the same
way.

\begin{lemma}
\label{lem:multiple_class}
Let $\shl$ be a coherent $\Dh[X]$-module without $\h$-torsion.
Then, for any $i > 0$, the $\shd_X$-module $\shl/\h^i\shl$ is
coherent, and we have the equality $[\shl/\h^i\shl] = i\cdot
[\gr(\shl)]$ in $\K(\mdc[\shd_X])$.
\end{lemma}

\begin{proof}
Since the functor $(\scdot) \tens[\coro] \coro/\h^i\coro$ is right
exact,
$\shl/\h^i\shl$ is a coherent $\shd_X$-module.  Since $\shl$
has no $\h$-torsion, multiplication by $\h^i$ induces an isomorphism
$\shl/\h\shl \isoto \h^i\shl/\h^{i+1}\shl$.  We conclude by
induction on $i$ with the exact sequence
\eqn
&&0 \to \h^i\shl/\h^{i+1}\shl \to \shl/\h^{i+1}\shl
\to \shl/\h^i\shl \to 0.
\eneqn
\end{proof}

\begin{lemma}\label{lem:K}
For $\shm\in \mdc[{\Dhl[X]}]$, $U\subset X$ an open set and
$\shl\subset\shm|_U$ a lattice of $\shm|_U$, the class
$[\grh(\shl)] \in \K(\mdc[\shd_X|_U])$ only depends on $\shm$.
This defines a morphism of abelian sheaves $\shk(\Dhl[X]) \to
\shk(\shd_X)$.
\end{lemma}

\begin{proof}
(i) We first prove that $[\grh(\shl)]$ only depends on $\shm$.  We
consider another lattice $\shl'$ of $\shm|_U$.  Since $\shl$ is
a $\Dh[X]$-module of finite type, and $\shl'$ generates $\shm$,
there exists $n>1$ such that $\shl\subset\h^{-n}\shl'$.
Similarly, there exists $m>1$ with $\shl'\subset\h^{-m}\shl$, so
that we have the inclusions
\eqn
&&\h^{m+n+2}\shl \subset \h^{m+n+1}\shl
\subset \h^{m+1}\shl' \subset \h^{m}\shl'\subset \shl .
\eneqn
Any inclusion $A\subset B \subset C$ yields an identity
$[C/A]=[C/B]+[B/A]$ 
in the Grothendieck group, and we obtain in particular:
\begin{align*}
[\h^{m}\shl'/\h^{m+n+1}\shl]
&=[\h^{m}\shl'/\h^{m+1}\shl']+[\h^{m+1}\shl'/\h^{m+n+1}\
shl]\\
[\shl/\h^{m+n+1}\shl]&=[\shl/\h^{m+1}\shl']+[\h^{m+1}\shl'/\h^{m+n+1}\
shl]\\
[\shl/\h^{m+n+2}\shl]&=[\shl/\h^{m+1}\shl']+[\h^{m+1}\shl'/\h^{m+n+2}\
shl].
\end{align*}
Since our modules have no $\h$-torsion, we have isomorphisms of the
type $\h^k \shm_1 / \h^k \shm_2 \simeq \shm_1 / \shm_2$.  Then
Lemma~\ref{lem:multiple_class} and the above equalities give:
\begin{align*}
[\shl'/\h^{n+1}\shl]&= [\grh(\shl')]+[\shl'/\h^{n}\shl]\\
(m+n+1)[\grh(\shl)]&=[\shl/\h^{m+1}\shl']+[\shl'/\h^{n}\shl]\\
(m+n+2)[\grh(\shl)]&=[\shl/\h^{m+1}\shl']+[\shl'/\h^{n+1}\shl].
\end{align*}
A suitable combination of these lines gives
$[\grh(\shl)]=[\grh(\shl')]$, 
as desired.

\vspace{0.2cm}
\noindent
(ii) Now we consider an open subset $V\subset X$ and
$\shm\in\mdc[{\Dhl[X]}|_V]$. 
We choose an open covering $\{U_i\}_{i\in I}$ of $V$ such that for each
$i\in I$
$\shm\vert_{U_i}$ admits a lattice, say $\shl^i$.  We 
have seen that $[\grh(\shl^i)] \in \K(\mdc[\shd_X|_{U_i}])$ only
depends on $\shm$. This implies that 
\eqn
&&
[\grh(\shl^i)]|_{U_{i,j}}=[\grh(\shl^j)]|_{U_{i,j}}
\mbox{ in }\K(\mdc[\shd_X|_{U_{i,j}}]).
\eneqn
Hence the $[\grh(\shl^i)]$'s define a section, say $c(\shm)$, of
$\shk(\shd_X)$ over $V$.  By Lemma~\ref{lem:additivity}, $c(\shm)$
only depends on the class $[\shm]$ in $\K(\mdc[{\Dhl[X]}|_V])$, and
$\shm \mapsto c(\shm)$ induces the morphism
$\shk(\Dhl[X])\to\shk(\shd_X)$. 
\end{proof}

By Lemma~\ref{lem:K}, the following definition is well posed.

\begin{definition}
Let $\shm$ be a coherent $\Dhl[X]$-module. 
For $\shl\in\mdc[{\Dh[X]}]$ a (local) lattice, the
characteristic variety of $\shm$ is defined by 
\eqn 
&&\chv_{\h,\loc}(\shm) = \chv_\h(\shl).
\eneqn 
For
$\shm\in \Derb_\coh(\Dhl)$, one sets
$\chv_{\h,\loc}(\shm)=\bigcup_j\chv_{\h,\loc}(H^j(\shm))$.
\end{definition}

\begin{proposition}
The characteristic variety $\chv_{\h,\loc}(\scbul)$ is additive both on
$\mdcoh[\Dhl]$ and on $\Derb(\Dhl)$.
\end{proposition}

\begin{proof}
This follows from Proposition~\ref{prop:add_var_car}~(ii) and 
Lemma~\ref{lem:additivity}.
\end{proof}

Consider the functor
\eqn
&&\Sol_{\h,{\rm loc}}(\scbul)\cl\Derb(\Dhl)^\rop\to\Derb(\cor_X), 
\quad\shm\mapsto \rhom[{\Dhl}](\shm,\Ohl).      
\eneqn

\begin{proposition}\label{pr:SSloc}
Let $\shm\in\Derb_\coh(\Dhl)$. Then 
\[
\SSi\bl\Sol_{\h,{\rm loc}}(\shm)\br \subset \chv_{\h,\loc}(\shm).
\]
\end{proposition}

\begin{proof}
By d{\'e}vissage, we can assume that $\shm\in \mdc[{\Dhl[X]}]$. Moreover,
since the problem is local, we may assume that $\shm$ admits a lattice $\shl$.

One has the isomorphism
$\Sol_{\h,{\rm loc}}(\shm) \simeq \rhom[{\Dh}](\shl,\Ohl)$ by extension of scalars.
Taking a local resolution of $\shl$ by
free $\Dh$-modules of finite type, we deduce that $\Sol_{\h,{\rm loc}}(\shm)
\simeq F^\loc$ for $F=\Solh(\shl)$. The statement follows by \eqref{eq:SSloc} and Corollary~\ref{cor:SSsol}.
\end{proof}

One says that $\shm$ is holonomic if its characteristic variety is isotropic.

\begin{proposition}
Let $\shm\in\Derb_\hol(\Dhl)$.
Then $\Sol_{\h,{\rm loc}}(\shm)\in\Derb_\Cc(\cor_X)$.
\end{proposition}

\begin{proof}
By the same arguments and with the same notations as in the proof of Proposition~\ref{pr:SSloc}, 
we reduce to the case 
$\Sol_{\h,{\rm loc}}(\shm) \simeq F^\loc$, for $F=\Solh(\shl)$ and $\shl$ a lattice of 
$\shm\in\Mod_\hol(\Dhl)$. Hence $\shl$ is a holonomic $\Dh$-module, and $F\in\Derb_\Cc(\coro_X)$.
\end{proof}

\begin{remark}
In general the functor
\[
\Sol_{\h,{\rm loc}}\cl\Derb_\hol(\Dhl)^\rop\to\Derb_\Cc(\cor_X)
\]
is not locally essentially surjective.
In fact, consider the quasi-commutative diagram of categories
\eqn
&&\xymatrix{
\Derb_\hol(\Dh)^\rop \ar[rr]^-{\Solh}
\ar[d]_-{(\scdot)^{\rm loc}} 
&&\Derb_\Cc(\coro_X)  \ar[d]_-{(\scdot)^{\rm loc}}\\
\Derb_\hol(\Dhl)^\rop \ar[rr]^-{\Sol_{\h,{\rm loc}}} 
&&\Derb_\Cc(\cor_X).
}
\eneqn
By the local existence of lattices the left vertical arrow is locally 
essentially surjective. If $\Sol_{\h,{\rm loc}}$ were also locally
essentially surjective, so
should be the right vertical arrow. The following example shows that it
is not the case.

\begin{example}
Let $X=\C$, $U = X \setminus \{0\}$ and denote by $j\colon
U\hookrightarrow X$ the embedding. Set
$F=\reim j L$, where $L$ is the local system on $U$ with
stalk $\cor$ and monodromy $\h$. 
There is no $F_0\in\Derb_\Cc(\coro_X)$ such that $F \simeq (F_0)^{\rm
loc}$.
\end{example}

One can interpret this phenomenon by remarking that $\Derb_\hol(\Dhl)$
is equivalent to the localization of the category $\Derb_\hol(\Dh)$
with respect to the morphism $\h$, contrarily to the category
$\Derb_\Cc(\cor_X)$.
\end{remark}

\section{Links with deformation quantization}

In this last section, we shall briefly explain how the study of
deformation quantization algebras on complex symplectic manifolds
is related to $\Dh$. We follow the terminology of \cite{KS08}.

\medskip
The cotangent bundle $\stx=T^*X$ to the complex manifold $X$ has a
structure of a complex symplectic manifold and is endowed with the
$\coro$-algebra $\HW[\stx]$, a non homogeneous version of the algebra of
microdifferential operators.
Its subalgebra $\HWo[\stx]$ of operators of order at most zero is a
deformation quantization algebra.
In a system $(x,u)$ of local symplectic coordinates, 
$\HWo[\stx]$ is identified with the star algebra $(\Oh[\stx],\star)$
in which  the star product is given by the Leibniz product:
\eq\label{eq:wstar}
f\star g
&=&\sum_{\alpha\in\N^n} \dfrac{\h^{\vert\alpha\vert}}{\alpha !} 
(\partial^{\alpha}_uf)(\partial^{\alpha}_xg), \quad
\text{for }f,g\in\OO[\stx].
\eneq
In this section we will set for short
$\shaw \seteq \HWo[\stx]$,
so that $\shaw^{\rm loc} \simeq \HW[\stx]$.
Note that $\shaw$ satisfies Assumption~\ref{as:DQring}.

Let us identify $X$ with the zero section of the cotangent bundle
$\stx$.
Recall that $X$ is a local model for any smooth Lagrangian submanifold
of $\stx$,
and that $\Oh$ is a local model of any simple $\shaw$-module along
$X$.
As $\Oh$ has both a $\Dh$-module and an $\shaw|_X$-module structure,
there are morphisms of $\coro$-algebras
\begin{equation}\label{eq:actions}
\Dh \to \shend_{\coro}(\Oh) \from \shaw|_X.
\end{equation}

\begin{lemma}\label{lem:W0D}
The morphisms in \eqref{eq:actions}
are injective and induce an embedding $\shaw|_X\hookrightarrow\Dh$.
\end{lemma}

\begin{proof}
Since the problem is local, we may choose a 
local symplectic coordinate system $(x,u)$ on $\stx$ such that
$X=\{u=0\}$.
Then $\shaw|_X$ is identified with 
$\Oh[\stx]|_X$. As the action of $u_i$ on 
$\Oh$ is given by $\h \partial_{x_i}$, the morphism
$\shaw|_X\to \shend_{\coro}(\Oh)$ factors through $\Dh$,
and the induced morphism $\shaw|_X \to \Dh$ is described by
\eq\label{eq:AinDh}
\sum_{i\in\N} f_i(x,u) \h^i\mapsto 
\sum_{j\in\N}\left( \sum_{\alpha\in\N^n,\ |\alpha|\leq j}
\partial_u^\alpha f_{j-|\alpha|}(x,0)
\partial_x^\alpha \right) \h^j,
\eneq
which is clearly injective.
\end{proof}

Consider the following subsheaves of $\Dh$
\[
\D^{\h,m}  = \prod_{i\geq 0} \left( F_{i+m}\D\right) \h^i,\quad
\Dhf = \bigcup_{m\geq 0} \D^{\h,m}.
\]
Note that $\Dho$ and $\Dhf$ are subalgebras of $\Dh$, that $\Dho$ is
$\h$-complete while $\Dhf$ is not and that $\Dhol\simeq\Dhfl$.
By \eqref{eq:AinDh}, the image of $\sha|_X$ in $\Dh$ is contained in
$\Dho$.
(The ring $\Dho$ should be compared with the ring $\shr_{X\times\C}$ of
\cite{Sa05}.)

\begin{remark}
More precisely, denote by
$\Oh[\stx]{\hat|}_X\simeq(\OO[\stx]{\hat|}_X)^\h$ the formal 
restriction of $\Oh[\stx]$ along the submanifold $X$. 
Then the star product
in \eqref{eq:wstar} extends to this sheaf, and \eqref{eq:AinDh} induces
an isomorphism
$(\Oh[\stx]{\hat|}_X,\star)\simeq\Dho$.
\end{remark}

Summarizing, one has the compatible embeddings of algebras
\[
\xymatrix{
\shaw^{\rm loc}|_X \ar@{^(->}[r] & \Dhol \ar@{-}[r]^\sim & \Dhfl
\ar@{^(->}[r] & \Dhl \\
\shaw|_X \ar@{^(->}[u] \ar@{^(->}[r] & \Dho \ar@{^(->}[r] \ar@{^(->}[u]
& \Dhf \ar@{^(->}[r] \ar@{^(->}[u] & \Dh \ar@{^(->}[u]
}
\]
One has
\[
\gr \shaw|_X \simeq \OO[\stx]|_X, \quad
\gr     \Dho \simeq \OO[\stx]{\hat|}_X, \quad
\gr     \Dhf \simeq \gr \Dh \simeq \D.
\]

\begin{proposition}\label{pr:DhoAflat}
\bnum
\item The algebra $\Dho$ is faithfully flat over $\shaw|_X$.
\item The algebra $\Dhl$ is flat over $\shaw^{\rm loc}|_X$.
\enum
\end{proposition}

\begin{proof}
(i) follows from Theorem~\ref{th:flat2}. 

\noindent
(ii) follows from (i) and the isomorphism $(\Dho)^\loc\simeq\Dhl$.
\end{proof}

The next examples show that the scalar extension functor
\[
\mdc[\Dho]\to\mdc[\Dh]
\]
is neither exact nor full.

\begin{example}
Let $X=\C^2$ with coordinates $(x,y)$. Then $\h\partial_y$ is injective
as an endomorphism 
of $\Dho/\langle\h\partial_x\rangle$ but it is not injective as an
endomorphism 
of $\Dh/\langle\h\partial_x\rangle$, since $\partial_x$ belongs to its
kernel. This shows that $\Dh$ 
is not flat over $\Dho$.
\end{example}

\begin{example}\label{exa:WtoDnotff}
This example was communicated to us by Masaki Kashiwara.
Let $X=\C$ with coordinate $x$, and denote by $(x,u)$ the symplectic
coordinates on $\stx=T^*\C$. Consider the cyclic $\shaw$-modules
\[
\shm=\shaw/\langle x-u\rangle,\quad \shn=\shaw/\langle x\rangle,
\]
and their images in $\md[\Dh]$
\[
\shm'=\Dh/\langle x-\h\partial_x\rangle,
\quad \shn'=\Dh/\langle x\rangle.
\]
As their supports in $\stx$ differ, $\shm$ and $\shn$  are not
isomorphic as $\shaw$-modules.
On the other hand, in $\Dh$ one has the relation
\begin{equation}\label{eq:ed2}
x \cdot e^{\h\partial_x^2/2} = e^{\h\partial_x^2/2} \cdot
(x-\h\partial_x),
\end{equation}
and hence an isomorphism $\shm'\isoto \shn'$ given by $[P]\mapsto
[P\cdot e^{-\h\partial_x^2/2}]$.
In fact, one checks that
\[
\hom[{\shaw}](\shm,\shn)|_X = 0, \quad
\hom[\Dh](\shm',\shn') \simeq \coro_X.
\]
\end{example}

\appendix

\section{Complements on constructible sheaves}

Let us review some results, well-known from 
the specialists (see {\em e.g.,}\cite[Proposition~3.10]{SSc94}), but which are
usually
stated over a field, and we need to work here over the ring $\coro$. 

Let $\cora$ be a commutative 
unital Noetherian ring of finite global dimension. Assume that $\cora$
is syzygic, i.e.~that 
any finitely generated $\cora$-module admits a finite projective
resolution by finite free modules.
(For our purposes we will either have $\cora=\C$ or $\cora=\coro$).

Let $X$ be a real analytic manifold. Denote by $\mdrc[\cora_X]$
the abelian category of $\R$-constructible sheaves on $X$ and by 
$\Derb_\Rc(\cora_X)$ the bounded derived category of
sheaves of $\cora$-modules with $\R$-constructible cohomology.

For the next two lemmas we recall some notations and results
of~\cite{Ka84,KS90}. We consider a simplicial complex $\SC = (S,\Delta)$,
with set of vertices $S$ and set of simplices $\Delta$.  We let
$|\SC|$ be the realization of $\SC$. Thus $|\SC|$ is the disjoint
union of the realizations $|\sigma|$ of the simplices.  For a simplex
$\sigma\in \Delta$, the open set $U(\sigma)$ is defined in
\cite[(8.1.3)]{KS90}.  A sheaf $F$ of $\cora$-modules on $|\SC|$ is said
weakly
$\SC$-constructible if, $\forall \sigma \in \Delta$, $F|_{|\sigma|}$
is constant.  An object $F\in \Derb(\cora_{|\SC|})$ is said weakly
$\SC$-constructible if its cohomology sheaves are so.  If moreover, all
stalks $F_x$ are perfect complexes, $F$ is said $\SC$-constructible.
By~\cite[Proposition~8.1.4]{KS90} we have isomorphisms, for a weakly
$\SC$-constructible sheaf $F$ and for any $\sigma\in\Delta$ and $x\in
|\sigma|$:
\begin{gather}\label{eq:section_simplex}
\sect(U(\sigma);F) \isoto \sect(|\sigma|;F) \isoto F_x, \\
\label{eq:section_simplex2}
H^j(U(\sigma);F) = H^j(|\sigma|;F) = 0, \quad \text{for $j \not= 0$.}
\end{gather}
It follows that, for a weakly $\SC$-constructible $F\in
\Derb(\cora_{|\SC|})$, the natural morphisms of complexes of
$\cora$-modules
\begin{equation}
\label{eq:section_simplex_qis}
\sect(U(\sigma);F) \to \sect(|\sigma|;F) \to F_x
\end{equation}
are quasi-isomorphisms.

For $U\subset X$ an open subset, we denote by $\cora_U\eqdot(\cora_X)_U$
the
extension by $0$ of the constant sheaf on $U$. 

\begin{proposition}\label{pro:resol1}
Let $F\in \Derb_\Rc(\cora_X)$. Then
\bnum
\item $F$ is isomorphic to a complex
\eqn 
&&0\to\bigoplus_{i_a\in I_a} \cora_{U_{a,i_a}}
\to\cdots\to \bigoplus_{i_b\in I_b}\cora_{U_{b,i_b}}\to 0, 
\eneqn 
where the $\{U_{k,i_k}\}_{k,i_k}$'s are locally finite families of
relatively compact subanalytic open subsets of $X$.
\item $F$ is isomorphic to a complex
\eqn 
&&0\to\bigoplus_{i_a\in I_a} \sect_{V_{a,i_a}}\cora_X
\to\cdots\to \bigoplus_{i_b\in I_b}\sect_{V_{b,i_b}}\cora_X\to 0, 
\eneqn 
where the $\{V_{k,i_k}\}_{k,i_k}$'s are locally finite families of
relatively compact subanalytic open subsets of $X$.
\enum
\end{proposition}

\begin{proof}
(i) By the triangulation theorem for subanalytic sets (see for
example~\cite[Proposition~8.2.5]{KS90}) we may assume that $F$ is an
$\SC$-constructible object in $\Derb(\cora_{|\SC|})$ for some
simplicial complex $\SC = (S,\Delta)$.  For $i$ an integer, let
$\Delta_i\subset \Delta$ be the subset of simplices of dimension
$\leq i$ and set $\SC_i = (S,\Delta_i)$.  We denote by $\Kb(\cora)$
(resp.  $\Kb(\cora_{|\SC|})$) the category of bounded complexes of
$\cora$-modules (resp.  sheaves of $\cora$-modules on $|\SC|$) with
morphisms up to homotopy.  We shall prove by induction on $i$ that
there exists a morphism $u_i\cl G_i\to F$ in $\Kb(\cora_{|\SC|})$
such that:
\begin{itemize}
\item [(a)] the $G_i^k$ are finite direct sums of
$\cora_{U(\sigma_\alpha)}$'s for some $\sigma_\alpha\in\Delta_i$,
\item [(b)] $u_i\vert_{|\SC_i|}\cl G_i\vert_{|\SC_i|}\to
F\vert_{|\SC_i|}$ is a quasi-isomorphism.
\end{itemize}
The desired result is obtained for $i$ equal to the dimension of $X$.

\vspace{0.2cm}
\noindent
(i)-(1) For $i=0$ we consider $F|_{|\SC_0|}\simeq
\bigoplus_{\sigma\in\Delta_0} F_\sigma$.  
The complexes
$\sect(U(\sigma);F)$, $\sigma\in \Delta_0$, have finite bounded
cohomology by the quasi-isomorphisms \eqref{eq:section_simplex_qis}. 
Hence we may choose bounded complexes of finite free
$\cora$-modules, $R_{0,\sigma}$, and morphisms $u_{0,\sigma} \cl
R_{0,\sigma} \to \sect(U(\sigma);F)$ which are quasi-isomorphisms.

We have the natural isomorphism $\sect(U(\sigma);F) \simeq a_*
\hom[\Kb(\cora_{|\SC|})](\cora_{U(\sigma)}, F)$ in $\Kb(\cora)$, where
$a\cl |\SC| \to \rmpt$ is the projection and $\hom$ is the internal
$\Hom$ functor.  We deduce the adjunction formula, for $R\in
\Kb(\cora)$, $F\in \Kb(\cora_{|\SC|})$:
\begin{equation}
\label{eq:adj_Kb}
\Hom[\Kb(\cora)](R, \sect(U(\sigma);F) )
\simeq
\Hom[\Kb(\cora_{|\SC|})](R_{U(\sigma)}, F).
\end{equation}
Hence the $u_{0,\sigma}$ induce $u_0 \cl G_0 \eqdot
\bigoplus_{\sigma \in \Delta_0} (R_{0,\sigma})_{U(\sigma)} \to F$.
By~\eqref{eq:section_simplex_qis} $(u_0)_x$ is a quasi-isomorphism for
all $x\in |\SC_0|$, so that $u_0\vert_{|\SC_0|}$ also is a
quasi-isomorphism, as required.

\vspace{0.2cm}
\noindent
(i)-(2) We assume that $u_i$ is built and let $H_i = M(u_i)[-1]$ be the
mapping cone of $u_i$, shifted by $-1$. By the distinguished triangle
in $\Kb(\cora_{|\SC|})$
\begin{equation}
\label{eq:dist_tr_1}
H_i \xto{v_i} G_i \xto{u_i} F \xto{+1}
\end{equation}
$H_i|_{|\SC_i|}$ is quasi-isomorphic to $0$. Hence $\bigoplus_{\sigma
\in \Delta_{i+1} \setminus \Delta_i} (H_i)_{|\sigma|} \to
H_i|_{|\SC_{i+1}|}$ is a quasi-isomorphism.  As above we choose
quasi-isomorphisms $u_{i,\sigma} \cl R_{i+1,\sigma} \to
\sect(U(\sigma);H_i)$, $\sigma\in \Delta_{i+1} \setminus \Delta_i$,
where the $R_{i+1,\sigma}$ are bounded complexes of finite free
$\cora$-modules.  By~\eqref{eq:adj_Kb} again the $u_{i,\sigma}$ induce
a morphism in $\Kb(\cora_{|\SC|})$
$$
u'_{i+1}\cl G'_{i+1} \eqdot 
\bigoplus_{\sigma \in \Delta_{i+1}\setminus \Delta_i}
(R_{i+1,\sigma})_{U(\sigma)}   \to H_i .
$$
For $x\in |\SC_{i+1}| \setminus |\SC_i|$, $(u'_{i+1})_x$ is a
quasi-isomorphism by~\eqref{eq:section_simplex_qis}, and, for $x\in
|\SC_i|$, this is trivially true. Hence $u'_{i+1} |_{|\SC_{i+1}|}$ is
a quasi-isomorphism.

Now we let $H_{i+1}$ and $G_{i+1}$ be the mapping cones of $u'_{i+1}$
and $v_i \circ u'_{i+1}$, respectively. We have distinguished
triangles in $\Kb(\cora_{|\SC|})$
\begin{equation}
\label{eq:dist_tr_2}
G'_{i+1} \xto{u'_{i+1}} H_i \to H_{i+1} \xto{+1},
\qquad
G'_{i+1} \xto{v_i \circ u'_{i+1}} G_i \to G_{i+1} \xto{+1}.
\end{equation}
By the contruction of the mapping cone, the definition of $G'_{i+1}$
and the induction hypothesis, $G_{i+1}$ satifies property (a) above.
The octahedral axiom applied to triangles~\eqref{eq:dist_tr_1}
and~\eqref{eq:dist_tr_2} gives a morphism $u_{i+1} \cl G_{i+1} \to F$
and a distinguished triangle $H_{i+1} \to G_{i+1} \xto{u_{i+1}} F
\xto{+1}$. By construction $H_{i+1}|_{|\SC_{i+1}|}$ is
quasi-isomorphic to $0$ so that $u_{i+1}$ satifies property (b) above.

(ii) Consider the duality functor $\RD'_\cora(\scdot) =
\rhom[\cora_X](\scdot,\cora_X)$.
Set $G=\RD'_\cora(F)$, and represent it by a bounded complex as in (i).
Since $U_{k,i_k}$ corresponds to an open subset of the form $U(\sigma)$
in $|\SC|$, the sheaves $\cora_{U_{k,i_k}}$ are acyclic for the functor
$\RD'_\cora$. Hence $F\simeq\RD'_\cora(G)$ can be represented as
claimed.
\end{proof}

\begin{lemma}\label{le:surj_sections}
Let $F\to G \to
0$ be an exact sequence in $\mdrc[\cora_X]$.  Then for any relatively
compact subanalytic open subset $U\subset X$, there exists a finite
covering $U= \bigcup_{i\in I} U_i$ by subanalytic open subsets such
that, for each $i\in I$, the morphism $F(U_i) \to G(U_i)$ is
surjective.
\end{lemma}
\begin{proof}
As in the proof of Proposition~\ref{pro:resol1} we may assume that $F$
and
$G$ are constructible sheaves on the realization of some finite
simplicial complex $(S,\Delta)$.  For $\sigma\in \Delta$ the
morphism $\sect(U(\sigma);F) \to\sect(U(\sigma);G)$ is surjective,
by~\eqref{eq:section_simplex}.  Since $|\SC |$ is the finite union
of the $U(\sigma)$ this proves the lemma.
\end{proof}

\section{Complements on subanalytic sheaves}\label{ap:suban}

We review here some well-known results (see \cite[Chapter~7]{KS01} and
\cite{Pr05}) 
but which are usually stated over a field, and we need to work here over
the ring $\coro$. 

Let $\cora$ be a commutative 
unital Noetherian ring of finite global dimension
(for our purposes we will either have $\cora=\C$ or $\cora=\coro$).
Let $X$ be a real analytic manifold, and consider the natural morphism
$\rho \cl X\to\Xsa$ to the associated subanalytic site.

\begin{lemma} \label{lem:rhoexact}
The functor $\oim{\rho}\cl \mdrc[\cora_X]\to
\md[\cora_\Xsa]$ is exact and $\opb{\rho}\oim{\rho}$ is isomorphic to
the
canonical functor $\mdrc[\cora_X]\to\md[\cora_X]$.
\end{lemma}

\begin{proof}
Being a direct image functor, $\oim{\rho}$ is left exact. It is right
exact thanks to Lem\-ma~\ref{le:surj_sections}. The composition
$\opb{\rho}\oim{\rho}$ is isomorphic to the identity on $\md[\cora_X]$
since the open sets of the site $\Xsa$ give a basis of the
topology of $X$.
\end{proof}

In the sequel, we denote by $\mdrc[\cora_\Xsa]$ the image by the
functor $\oim{\rho}$ of $\mdrc[\cora_X]$ in $\md[\cora_\Xsa]$.  Hence
$\oim{\rho}$ induces an equivalence of categories
$\mdrc[\cora_X]\simeq\mdrc[\cora_\Xsa]$.  We also denote by
$\Derb_\Rc(\cora_{\Xsa})$ the full triangulated subcategory of
$\Derb(\cora_{\Xsa})$ consisting of objects with cohomology in
$\mdrc[\cora_\Xsa]$.

\begin{corollary}
\label{cor:thick}
The subcategory $\mdrc[\cora_\Xsa]$ of $\md[\cora_\Xsa]$ is thick.
\end{corollary}

\begin{proof}
Since $\oim{\rho}$ is fully faithful and exact, $\mdrc[\cora_\Xsa]$
is stable by kernel and cokernel.  It remains to see that, for $F,G
\in \mdrc[\cora_X]$
$$
\Ext[{\mdrc[\cora_X]}]{1}(F,G)
\simeq  \Ext[{\md[\cora_\Xsa]}]{1}(\oim{\rho}F,\oim{\rho}G).
$$
By~\cite{Ka84} we know that the first $\Ext{1}$ may as well be
computed in $\md[\cora_X]$. We see that the functors $\rho^{-1}$ and
$\roim{\rho}$ between $\Derb(\cora_X)$ and $\Derb(\cora_{\Xsa})$ are
adjoint, and moreover $\rho^{-1}\roim{\rho} \simeq \id$.  Thus, for
$F',G' \in \Derb(\cora_X)$ we have
$$
\Hom[{\Derb(\cora_{\Xsa})}](\roim{\rho}F',\roim{\rho}G')
\simeq \Hom[{\Derb(\cora_X)}](F',G'),
$$
and this gives the result.
\end{proof}

This corollary gives the equivalence $\Derb_\Rc(\cora_X) \simeq
\Derb_\Rc(\cora_{\Xsa})$, both categories being equivalent to
$\Derb(\mdrc[\cora_X])$.

\providecommand{\bysame}{\leavevmode\hbox to3em{\hrulefill}\thinspace}

{\small
\noindent
Andrea D'Agnolo:
Universit{\`a} degli Studi di Padova,
Dipartimento di Ma\-te\-ma\-ti\-ca Pura ed Applicata,
via Trieste 63,
35121 Padova, Italy\\
E-mail address: {\tt dagnolo@math.unipd.it}\\
Web page: {\tt www.math.unipd.it/\textasciitilde dagnolo}

\medskip\noindent
St{\'e}phane Guillermou:
Institut Fourier,
Universit{\'e} de Grenoble I,
BP 74, 38402
Saint-Martin d'H{\`e}res, France\\
E-mail address: {\tt Stephane.Guillermou@ujf-grenoble.fr}\\
Web page: {\tt www-fourier.ujf-grenoble.fr/\textasciitilde guillerm}

\medskip\noindent
Pierre Schapira:
Institut de Math{\'e}matiques,
Universit{\'e} Pierre et Marie Curie,
175 rue du Chevaleret,
75013 Paris, France\\
E-mail address: {\tt schapira@math.jussieu.fr}\\
Web page: {\tt www.math.jussieu.fr/\textasciitilde schapira}
}
\end{document}